\documentclass[final,12pt]{colt2024} 

\title[On Finding Small Hyper-Gradients in Bilevel Optimization]{On Finding Small Hyper-Gradients in Bilevel Optimization: \\
 Hardness Results and Improved Analysis}
\usepackage{times}
\usepackage{xcolor}
\usepackage{multirow}
\usepackage{multicol}
\usepackage{bbding}
\usepackage{makecell}
\graphicspath{{media/}}     
\usepackage{pifont}
\definecolor{mydarkgreen}{RGB}{39,130,67}
\definecolor{mydarkred}{RGB}{192,25,25}

\definecolor{bgcolor}{rgb}{0.93,0.99,1}
\usepackage{threeparttable}
\usepackage{wrapfig}

\usepackage{threeparttable}
\usepackage{thmtools}
\usepackage{thm-restate}
\usepackage{nicefrac}
\usepackage{enumitem}

\usepackage{amsmath}\allowdisplaybreaks
\usepackage{amsfonts,bm}
\usepackage{amssymb}
\usepackage{amsmath}
\usepackage{algorithm}
\usepackage{algorithmic}
\usepackage{multirow}
\usepackage{booktabs}
\usepackage{xcolor}

\def\eqref#1{equation~(\ref{#1})}
\def\Eqref#1{Equation~(\ref{#1})}

\def\1{\bf{1}}

\def\vone{{\bf{1}}}

\def\fA{{\mathcal{A}}}
\def\fB{{\mathcal{B}}}

\def\fL{{\mathcal{L}}}
\def\fM{{\mathcal{M}}}
\def\fN{{\mathcal{N}}}
\def\fO{{\mathcal{O}}}
\def\fP{{\mathcal{P}}}

\def\fR{{\mathcal{R}}}
\def\fS{{\mathcal{S}}}
\def\fT{{\mathcal{T}}}

\def\fX{{\mathcal{X}}}
\def\fY{{\mathcal{Y}}}

\def\sB{{\mathbb{B}}}

\def\BE{{\mathbb{E}}}

\def\BI{{\mathbb{I}}}

\def\BR{{\mathbb{R}}}

\DeclareMathOperator*{\argmin}{arg\,min}

\newtheorem{thm}{Theorem}[section]
\newtheorem{dfn}{Definition}[section]
\newtheorem{exmp}{Example}[section]
\newtheorem{lem}{Lemma}[section]
\newtheorem{asm}{Assumption}[section]
\newtheorem{rmk}{Remark}[section]

\newtheorem{prop}{Proposition}[section]

\makeatletter
\def\Ddots{\mathinner{\mkern1mu\raise\p@
\vbox{\kern7\p@\hbox{.}}\mkern2mu
\raise4\p@\hbox{.}\mkern2mu\raise7\p@\hbox{.}\mkern1mu}}
\makeatother

\makeatletter
\newcommand*{\rom}[1]{\expandafter\@slowromancap\romannumeral #1@}
\makeatother

\coltauthor{%
 \Name{Lesi Chen}${}^*$  \Email{chenlc23@mails.tsinghua.edu.cn}\\
 \addr {IIIS, Tsinghua University\\
       Shanghai Qizhi Institute
 }
 \AND
 \Name{Jing Xu}${}^*$ \Email{xujing21@mails.tsinghua.edu.cn}\\
 \addr {IIIS, Tsinghua University
 }
 \AND
 \Name{Jingzhao Zhang}${}^\dagger$ 
 \Email{jingzhaoz@mail.tsinghua.edu.cn}\\
 \addr {IIIS, Tsinghua University\\
        Shanghai AI Lab\\
       Shanghai Qizhi Institute
 }
}

\begin{document}

\maketitle 
\begingroup
\begin{NoHyper}
\renewcommand\thefootnote{${}^*$}
\footnotetext{Equal contributions.}
\end{NoHyper}
\begin{NoHyper}
\renewcommand\thefootnote{${}^\dagger$}
\footnotetext{The corresponding author.}
\end{NoHyper}
\endgroup
\begin{abstract}%
Bilevel optimization reveals the inner structure of otherwise oblique optimization problems, such as hyperparameter tuning, neural architecture search, and meta-learning. 
A common goal in bilevel optimization is to minimize a hyper-objective that implicitly depends on the solution set of the lower-level function. Although this hyper-objective approach is widely used, its theoretical properties have not been thoroughly investigated in cases where \textit{the lower-level functions lack strong convexity}. 
In this work, we first provide hardness results to show that the goal of finding stationary points of the hyper-objective for
nonconvex-convex bilevel optimization can be intractable for zero-respecting algorithms.
Then we study a class of tractable nonconvex-nonconvex bilevel problems when the lower-level function satisfies the Polyak-Łojasiewicz (PL) condition. We show a simple first-order algorithm can achieve complexity bounds of $\tilde{\mathcal{O}}(\epsilon^{-2})$,
$\tilde{\mathcal{O}}(\epsilon^{-4})$ and $\tilde{\mathcal{O}}(\epsilon^{-6})$ in the deterministic, partially stochastic, and fully stochastic setting respectively. 
The complexities in the first two cases are optimal up to logarithmic factors.
\end{abstract}


\section{Introduction}

The goal of bilevel optimization is to minimize the upper-level function $f(x,y)$ under the constraint that $y$ is minimized with respect to the lower-level function $g(x,y)$. Formally, it is defined as,
\begin{align} \label{prob:main}
    \min_{x \in \BR^{d_x},y \in Y^*(x) } f(x,y), \quad  Y^*(x) = \arg \min_{y \in \BR^{d_y}} g(x,y).
\end{align}
Bilevel optimization in this form has received increasing attention due to its wide applications in many machine learning problems, including hyperparameter  tuning~\citep{franceschi2018bilevel,pedregosa2016hyperparameter}, neural architecture search~\citep{liu2018darts,wang2022zarts,zoph2016neural,zhang2021idarts}, meta-learning~\citep{franceschi2018bilevel,hospedales2021meta,ravi2017optimization,pham2021contextual,rajeswaran2019meta}, out-of-distribution learning~\citep{zhou2022model},
adversarial training~\citep{goodfellow2020generative,sinha2017certifying,lin2020gradient,lin2020near}, reinforcement learning~\citep{konda1999actor,hong2023two}, causal learning~\citep{jiang2022invariant,arjovsky2019invariant}. 

The hyper-objective approach \citep{dempe2002foundations} reformulates Problem (\ref{prob:main}) by a minimization problem defined below,
\begin{align} \label{hyper-refo}
    \min_{x \in \BR^d} \varphi(x), \text{ where } \varphi(x) = \min_{y \in Y^*(x) } f(x,y)
\end{align}
is called the hyper-objective. When $\varphi(x)$ has Lipschitz continuous gradients, a common is to find almost stationary points of $\varphi(x)$.

 Finding stationary points can be especially easy when the lower-level function is strongly convex, because \Eqref{hyper-refo} can be simplified to the composite optimization problem below. Specifically, since $Y^*(x)$ has only one element when the lower-level function is strongly convex, we have $ Y^*(x) = \{y^*(x)\}$ and
\begin{align} \label{prob:sc}
    \min_{x \in \BR^{d_x}} \varphi(x):= f(x,y^*(x)), ~{\rm where}~ y^*(x) = \arg \min_{y \in \BR^{d_y}} g(x,y).
\end{align}
Further, the implicit function theorem~\citep{dontchev2009implicit} implies
\begin{align} \label{grad-sc}
    \nabla \varphi(x) = \underbrace{\nabla_x f(x,y^*(x))}_{\text{Explicit Gradient}} + \underbrace{\left( \nabla y^*(x) \right)^\top \nabla_y f(x,y^*(x))}_{\text{Implicit Gradient}}.
\end{align}
This equation enables one to estimate the hyper-gradient $\nabla \varphi(x)$ and perform gradient descent on $\varphi(x)$.
AID~\citep{ghadimi2018approximation} and ITD~\citep{ji2021bilevel} estimate 
$\nabla \varphi(x)$ with Hessian-vector-product oracles. 
The more recently proposed F${}^2$BA~\citep{chen2023near} estimate $\nabla \varphi(x)$  with gradient oracles.
All these methods 
can find a stationary point of $\varphi(x)$. 


However, due to the prevalence of nonconvex functions in real-world scenarios, the strong convexity assumption may limit the applicability of algorithms. Therefore, our work first aims to study whether convexity, as a relaxation of strong convexity, suffices for finding small hyper-gradients efficiently. Specifically, we hope to answer the question below.
\begin{center}
    \textit{Can we find stationary points  of $\varphi(x)$ when the lower level function $g(x,y)$ is (strictly) convex but not strongly convex in $y$?} 
\end{center}
We provide a negative answer to the above question. We first prove in 
Example \ref{exmp:disconti} that in the case where $g(x, \cdot)$ is merely convex, $\varphi(x)$ may not have stationary points since
$\varphi(x) $ may be discontinuous.
Furthermore, we show that the continuity of $\varphi(x)$ can be fully characterized by the Pompeiu–Hausdorff continuity of $Y^*(x)$
in Theorem \ref{thm:cont-cont}.



We then study the cases when $\nabla \varphi(x)$ exists, \textit{e.g.} when 
the lower-level function is strictly convex.  We demonstrate that the stationary points of $\varphi(x)$ may still be computationally hard for any zero-respecting algorithms. 
This algorithm class contains a broad range of existing algorithms, including all the algorithms mentioned in this paper.
\begin{thm}[Informal version of Theorem \ref{thm:NO}]
There exists a bilevel problem instance whose lower-level function is strictly convex and both $f,g$ satisfy regular smoothness conditions, such that any zero-respecting algorithm gets stuck at the initialization $x_0$.
\end{thm}




Given the negative result, we aim to study intermediate function classes that are not strongly convex but computationally tractable. In particular, we study the cases when the lower-level function satisfies the Polyak-Łojasiewicz (PL) conditions. This condition allows global nonconvexity but ensures local strong convexity uniformly in a subspace. 

The  PL condition can pose nontrivial challenges since we have neither \Eqref{prob:sc} nor (\ref{grad-sc}) in this case.
Researchers have provided novel analyses in this case.
\citet{xiao2023generalized} proposed a Hessian-vector-product-based method GALET with non-asymptotic convergence to KKT points of the gradient-based reformulated problem (\Eqref{grad-refo}) when $g$ is PL in $y$. 
\citet{kwon2023penalty} proved the differentiability of $\varphi(x)$ when the penalty function $\sigma f+g$ is uniformly PL in $y$ for all $\sigma$ in the neighborhood of zero, and 
showed a proximal variant of F${}^2$BA, which we call Prox-F${}^2$BA, can find a stationary point of $\varphi(x)$. Based on their differentiability result, we show that GALET also converges to a stationary point of $\varphi(x)$ in Appendix~\ref{apx:GALET}.


Although Prox-F${}^2$BA \citep{kwon2023penalty} has been shown to converge to an $\epsilon$-stationary point of $\varphi(x)$ with $\tilde \fO(\epsilon^{-3})$ first-order oracle calls,
the rate is worse than the $\fO(\epsilon^{-2})$ optimal rate of gradient descent on nonconvex single-level optimization.
Therefore, listed as an important future direction, \citet{kwon2023penalty} asked the question below.
\begin{center}
\textit{Can one achieve the (near)-optimal rate for nonconvex-PL bilevel problems with gradient oracles?}
\end{center}
We give a positive answer to this question. 
We show the F${}^2$BA~\citep{chen2023near} can already achieve this fast convergence rate with a sharp analysis.
Our improvement over \citet{kwon2023penalty} comes from establishing a tighter bound on the smoothness constant of $\varphi(x)$.


\begin{table*}[t]
    \centering
    \caption{
    We present the complexities of different methods for nonconvex-PL bilevel problems. 
   } 
    \label{tab:res}
    \begin{threeparttable}
    \begin{tabular}{c c c c c c}
    \hline 
    Oracle     & Method  &  Deterministic & \makecell[c]{Partially \\
    Stochastic } &
    \makecell[c]{Fully\\ Stochastic}
      & Reference  \\
    \hline \hline \addlinespace
    2nd & GALET~\tnote{\color{blue}(a)} & $\tilde \fO(\kappa^5 \epsilon^{-2})$ & - & - & \citet{xiao2023generalized} \\
    1st & Prox-F${}^2$BA~\tnote{\color{blue}(b)} & $\tilde \fO(\kappa^{p_1} \epsilon^{-3})$ &  $ \tilde \fO(\kappa^{p_2}  \epsilon^{-5})$ & $\tilde \fO(\kappa^{p_3} \epsilon^{-7})$ &   \citet{kwon2023penalty} \\
    1st & F${}^2$BA & $\tilde \fO(\kappa^4 \epsilon^{-2})$ & $ \tilde \fO(\kappa^6 \epsilon^{-4})$ &$\tilde \fO(\kappa^{12} \epsilon^{-6})$ & This Paper \\
     \hline
    \end{tabular}
    \begin{tablenotes}
    {\scriptsize   
     \item  {\color{blue} (a)}  Although \citet{xiao2023generalized} did not provide the dependency on $\kappa$ in the complexity, we can calculate by the way in our Remark \ref{rmk:comp-GALET}.
     Their analysis additionally requires the smallest singular value of $\nabla_{yy}^2 g(x,y)$ has a constant gap between zero, which makes $\nabla_{yy}^2 g(x,y)$ have a constant rank.
     \item {\color{blue} (b)} We use $p_1,p_2,p_3$ to denote the polynomial dependency in $\kappa$ since they are not provided by \citet{kwon2023penalty}.
    }
    \end{tablenotes}
    \end{threeparttable}
\end{table*}

\begin{thm}[Informal version of Theorem \ref{thm:YES}]
Under regular conditions as \citet{kwon2023penalty}, 
F${}^2$BA can provably find an $\epsilon$-stationary point of $\varphi(x)$ with $\tilde \fO(\epsilon^{-2})$ first-order oracle calls for nonconvex-PL bilevel problems.
\end{thm}

We also extend our analysis to stochastic bilevel problems when an algorithm only has access to a noisy estimator of $\nabla f$ and $\nabla g$. Under the partially stochastic setting when the noise is only in $\nabla g$, we prove the stochastic F${}^2$BA has the $\tilde \fO(\epsilon^{-4})$ first-order oracle complexity, which is also near-optimal for stochastic optimization~\citep{arjevani2023lower}. Under the more general fully stochastic setting when noise appears both in $\nabla f$ and $\nabla g$, we prove the method has an $\tilde \fO(\epsilon^{-6})$ complexity.

Compared with \citet{xiao2023generalized}, our deterministic F${}^2$BA achieves the same $\tilde \fO(\epsilon^{-2})$ rate without the assistance of Hessian-vector-product oracles. Our method also has a better dependency on $\kappa$ since we do not use a squared trick on the Hessian of $g$. And we additionally study stochastic problems that have not been studied by \citet{xiao2023generalized}. Compared with \citet{kwon2023penalty}, we strictly improve the complexities for both deterministic and stochastic cases under the same assumptions.
We compare our results with prior works in Table \ref{tab:res}, and leave a more detailed introduction of related works in Appendix \ref{apx:related}.




\paragraph{Notations.} Throughout this paper, we use $\Vert \cdot\Vert $ to denote the $\ell_2$-norm of a vector or the operator norm of a matrix. We use $\sB_{\delta}(z) = \{z': \Vert z' - z \Vert \le \delta \}$ to denote  the $\ell_2$-ball centered at $z$ with radius $\delta$. 
We use notation $\tilde \fO(\,\cdot\,)$ to hide logarithmic factors in notation $\fO(\,\cdot\,)$.
For a matrix $A$, we use $A^\dagger$ to denote the  Moore–Penrose inverse. Notations ${\rm Ker}(A) = \{ x:Ax=0\}$ and ${\rm Range}(A) = \{Ax \}$ denote the kernel space and range space of $A$, respectively. For a vector $v$, we use the subscript $v_{[j]}$ to denote its $j$-th coordinate.


\section{Preliminaries}

This section presents some basic definitions that are commonly used in optimization~\citep{nesterov2018lectures}.
To start with, the following definitions describe different orders of smoothness and levels of convexity for a function.
\begin{dfn}
We say an operator $\fT(x): \BR^{d_1} \rightarrow \BR^{d_2 \times d_3}$ is $C$-Lipschitz for some $C > 0$  if
\begin{align*}
    \Vert \fT(x) - \fT(x') \Vert \le C \Vert x - x'\Vert, \quad \forall  x, x' \in \BR^{d_1}.
\end{align*}
For a function $h(x): \BR^d \rightarrow \BR$, we say it has $L$-Lipschitz gradients (or it is $L$-smooth) if it is differentiable and $\nabla h(x)$ is $L$-Lipschitz; we say it has $\rho$-Lipschitz Hessians if it is twice differentiable and $\nabla^2 h(x)$ is $\rho$-Lipschitz.
\end{dfn}

\begin{dfn}
 We say a function $h(x): \BR^d \rightarrow \BR$ is $\mu$-strongly convex for some $\mu > 0$  if for any $x,x' \in \BR^d$ and $t \in (0,1)$, we have that
 \begin{align*}
     h(tx+(1-t)x') \le t f(x) + (1-t) h(x') - \frac{1}{2} \mu t(1-t) \Vert x - x' \Vert^2.
 \end{align*}
 We say $h(x)$ is convex if $\mu=0$.
\end{dfn}
As one relaxation of the above strong convexity condition, the Polyak-Łojasiewicz (PL) condition,  independently introduced by \citet{polyak1967general} and \citet{lojasiewicz1963topological}, is formally defined as follows.
\begin{dfn} \label{dfn:PL}
We say a function $h(x) : \BR^d \rightarrow \BR$ is $\mu$-PL for some $\mu > 0$ if it has a non-empty solution set and for any $ x \in \BR^d$ it holds that $\Vert \nabla h(x) \Vert^2 \ge 2 \mu \left( h(x) - \min_{x \in \BR^d } h(x) \right)$.
\end{dfn}
Compared to the strong convexity, the PL condition allows nonconvexity and multiple minima. 
The PL condition has wide applications in establishing global convergence of many nonconvex learning problems, including neural network training~\citep{charles2018stability,liu2022loss,hardt2016identity,li2018algorithmic} and optimal control ~\citep{fazel2018global}.

Recall our goal is to minimize the hyper-objective $\varphi(x)$. Since $\varphi(x)$ is typically nonconvex in bilevel optimization, the common goal is to find an $\epsilon$-stationary point, defined as follows.
\begin{dfn}
    We say $x$ is an $\epsilon$-stationary point of a differentiable function $\varphi(x)$ if $\Vert \nabla \varphi(x) \Vert \le \epsilon$.
\end{dfn}

\section{Negative Results for General Convex Lower-Level Functions} 

This section formally describes the challenges for bilevel optimization without lower-level strong convexity assumption. 
In Section \ref{sec:cont-cont}, we show that $\varphi(x)$ may not have stationary points and analyze the underlying reasons behind it. In Section \ref{sec:NO}, we demonstrate that even if a stationary point of $\varphi(x)$ exists, a zero-respecting algorithm may not be able to find it within a finite time.



\subsection{Stationary Points May not Exist } \label{sec:cont-cont}

The following example shows that when the lower-level function only has convexity, $\varphi(x)$ (\Eqref{hyper-refo}) can be discontinuous and has no stationary points.

\begin{exmp}[\citet{lucchetti1987existence}] \label{exmp:disconti}
Consider a bilevel problem as Problem (\ref{prob:main}) with $d_x = 1$, $d_y = 1$.
Let $f(x,y) = x^2 + y^2$, $g(x,y) = xy + I_{C}(y)$, where $I_{C}(\,\cdot\,)$ is the indicator function of the set $ \{ y: 0 \le y \le 1 \}$. In this example $g(x,y)$ is convex in $y$. But the hyper-objective $\varphi(x)$  is discontinuous at $x=0$, because $\lim_{x \rightarrow 0^+} \varphi(x) = 0$, $\lim_{x \rightarrow 0^-} \varphi(x) = 1$.
\end{exmp}

\begin{remark}
In the above example, the lower-level problem is a constrained optimization in $y$. We can also give a similar counter-example for unconstrained problems by replacing $I_C(y)$ with a smoothed surrogate
$h(y) = (y - 1)\BI[y \ge 1] - y\BI[y \le 0]$ and then letting $g(x, y) = xy + h(y)$.
\end{remark}
In this example, the discontinuity of $\varphi(x)$ comes from the discontinuity of $Y^*(x)$. Below, we prove that this statement and its reverse generally holds.
As $Y^*(x)$ is a set-valued mapping, we introduce the Hausdorff distance and use it to define different types of continuity.


\begin{dfn} \label{def:hausdorff}
The Hausdorff distance between two sets $S_1,S_2\subseteq\mathbb{R}^d$ is defined as
\begin{align*}
{\rm dist}(S_1,S_2) =
\max\left\{\sup_{x_1\in S_1}  \inf_{x_2\in S_2} \Vert x_1 - x_2 \Vert ,\sup_{x_2\in S_2}\inf_{x_1\in S_1} \Vert x_1 - x_2 \Vert\right\}.
\end{align*}  
We also denote ${\rm dist}(v,S) = {\rm dist}(\{v\},S) $ for $v\in \mathbb{R}^d$, $S\subseteq\mathbb{R}^d$.
\end{dfn}

\begin{dfn}
We say a set-valued mapping $S(x): \BR^{n} \rightrightarrows \BR^{m}$ is (Pompeiu–Hausdorff) continuous if for any $x\in \BR^{n}$ and any  $\epsilon>0$, there exists $\delta>0$, such that for any $x^{\prime} \in \BR^{n}$ satisfying $\|x^{\prime}-x\|\le \delta$, we have ${\rm  dist }(S(x),S(x^\prime))\le \epsilon$.
\end{dfn}

\begin{dfn} 
We say a set-valued mapping $S(x): \BR^{n} \rightrightarrows \BR^{m}$ is (Pompeiu–Hausdorff) locally Lipschitz if for any $x \in \BR^{n}$, there exists $\delta>0$ and $L>0$ such that for any $x^{\prime} \in \sB_{\delta}(x)$ we have ${\rm dist }(S(x),S(x^\prime))\le L \Vert x -x' \Vert$. We say $S(x)$ is (globally) Lipschitz it holds for $\delta \rightarrow \infty$.
\end{dfn}
Based on the definitions, we show the following connections between $\varphi(x)$ and $Y^*(x)$. 
\begin{restatable}{thm}{thmcontcont} \label{thm:cont-cont}
Suppose that for any given $x \in \BR^{d_x}$ the set $Y^*(x)$ is non-empty and compact.
\begin{enumerate}[label=(\alph*)]
\item If $f(x,y)$ and $Y^*(x)$ are continuous , then $\varphi(x)$ is continuous. 
\item 
Conversely, if  $\varphi(x)$ is continuous for any continuous $f(x,y)$, then $Y^*(x)$ is continuous.
\item If $f(x,y)$ and $Y^*(x)$ are locally Lipschitz , then $\varphi(x)$ is locally Lipschitz. 
\item 
Conversely, if  $\varphi(x)$ is locally Lipschitz for any locally Lipschitz function $f(x,y)$, then $Y^*(x)$ is locally Lipschitz.
    \item If $f(x,y)$ is $C_f$-Lipschitz and $Y^*(x)$ is $\kappa$-Lipschitz, then we have that $\varphi(x)$ is $C_{\varphi}$-Lipschitz with $ C_{\varphi}=(\kappa+1)C_f$. 
\item 
Conversely, if $\varphi(x)$ is $C_{\varphi}$-Lipschitz for any $C_f$-Lipschitz $f(x,y)$, then we know that $Y^*(x)$ is $\kappa$-Lipschitz with~$\kappa = C_{\varphi}/C_f$.
    \end{enumerate}
\end{restatable}
The theorem implies that continuity of the hyper-objective $\varphi(x)$ requires a strong assumption on the set of minima $Y^*(x)$ for $g$ and suggests that one would need local strong convexity of $g$ for the hyper-gradients to exist. However, we will see in the next subsection that even in such cases, finding a small hyper-gradient can be hard.

\subsection{Stationary Points May be Intractable to Find} \label{sec:NO}


\begin{algorithm*}[t]  
\caption{Zero-respecting algorithms for Problem (\ref{prob:sim-C})} \label{alg:class}
\begin{algorithmic}[1] 
\STATE \textbf{inputs:} initialization $x_0,y_0$, number of outer loops $T$, number of inner loops $K$ \\[1mm]
\STATE \textbf{for} $ t=0,\cdots, T-1$ \\[1mm]
\STATE \quad Generate $y_{t}^0$ such that 
$
{\rm supp} (y_t^0) \subseteq  \bigcup_{ 0 \le s < t,~0 \le k \le K} {\rm supp} \left(y_s^k\right)$. \\[1mm]
\STATE \quad  \textbf{for} $k = 0,\cdots,K-1$ \\[1mm]
\STATE \quad  \quad Generate $y_t^{k+1}$ such that $ {\rm supp}(y_t^{k+1}) \subseteq \bigcup_{0 \le i \le k, ~h \in \{f,g \} } {\rm supp} \left( \nabla_y h\left(x_t, y_t^i\right) \right)$. \\[1mm]
 \STATE  \quad \textbf{end for} \\[1mm]
\STATE  \quad Generate $ x_{t+1}$ such that $ {\rm supp}(x_{t+1}) \subseteq \bigcup_{0 \le s \le t} {\rm supp} \left( \nabla_x f(x_s,y_s^K) \right) $. \\[1mm]
\STATE \textbf{end for}
\end{algorithmic}
\end{algorithm*}

This subsection shows that even for nonconvex-strictly-convex bilevel problems where $\nabla \varphi(x)$ is guaranteed to exist, finding a point with a small hyper-gradient can still be intractable. We prove the negative result on the following simplified case of Problem (\ref{prob:main}) when the lower-level function does not depend on $x$:
\begin{align} \label{prob:sim-C}
    \min_{x \in \BR^{d_x}, y \in Y^*} f(x,y), \quad {\rm s.t.} \quad Y^* = \arg \min_{y \in \BR^{d_y}} g(y).
\end{align}
This problem is much simpler because the implicit gradient in \Eqref{grad-sc} disappears as $Y^*$ is a fixed set. However, we can show that this problem is hard enough for the following algorithm class.

\paragraph{Zero-respecting algorithms for bilevel problems.} 
We define an algorithm class that covers a wide range of existing algorithms designed for bilevel optimization.
This definition is inspired by the classical definitions in~\citet{nesterov2018lectures} but has an additional structure for bilevel problems (\ref{prob:sim-C}). We first recall the definition of ``zero-respecting'' algorithms.
\begin{dfn}[Definition 1 \citet{carmon2021lower}]
For a vector $v \in \BR^d$, we use 
${\rm supp}(v) = \{ j \in [d] : v_{[j]} \ne 0\}$ to denote its support.
We say an algorithm $\fA$ is zero-respecting to oracle $\mathbb{O}$ $:\mathbb{R}^d\to \mathbb{R}^d$ if the sequence $\{z_t \}$ generated by algorithm $\fA$ only explore coordinates in the support of the previous oracles, \textit{i.e.} the sequence satisfies ${\rm supp}(z_{t+1}) \subseteq  \bigcup_{0 \le s \le t} {\rm supp} \left( \mathbb{O} (z_s) \right)$.
\end{dfn}
We then define the zero-respecting algorithms for Problem (\ref{prob:sim-C}).
\begin{dfn}
    We say a (first-order) algorithm for Problem (\ref{prob:sim-C})  is zero-respecting if it has the procedure as Algorithm \ref{alg:class}. Such an algorithm consists of two loops: In the inner loop, it generates a sequence $\{ y_t^k\}_{k=1}^K$ that is zero-respecting to $\nabla_y f(x_t,y)$ and $ \nabla g(y)$ for a fixed $x_t$; In the outer loop, it generates a sequence $\{x_t\}_{t=1}^T$ that is zero-respecting to $\nabla_x f(x, y^K)$.
\end{dfn}

The above zero-respecting algorithm class for bilevel problems subsumes many known algorithms  when applied to Problem (\ref{prob:sim-C}), including: AID~\citep{ghadimi2018approximation}, ITD~\citep{ji2021bilevel}, GALET~\citep{xiao2023generalized}, (Prox)-F${}^2$BA~\citep{kwon2023penalty,kwon2023fully,chen2023near}, FdeHBO~\citep{yang2024achieving}, 
    BGS-Opt~\citep{arbel2022non}, BDA~\citep{liu2020generic}, BVFIM~\citep{liu2021value}, PDBO~\citep{sow2022primal}, 
    SLM~\citep{lu2023slm},
    LV-HBA~\citep{yao2024constrained}.

Below, we give a hard instance such that all these algorithms cannot find small hypergradients as they get stuck at the initialization $x_0$.

\begin{restatable}{thm}{ThmNO} \label{thm:NO}
Without loss of generality, suppose $x_0= y_0 = 0$ (otherwise we can translate the functions and the result still holds).
Fix $T$ and $K$. Let $d_x = 1$, $d_y = q = 2TK$, and
\begin{align*}
    f(x,y) = 2 (x +1)^2 \sum_{j=q/2}^{q} \psi(y_{[j]}), \quad g(y) = \frac{1}{8} (y_{[1]} -1/\sqrt{q} )^2 +  \frac{1}{8} \sum_{j=1}^{q-1} \left(y_{[j+1]} - y_{[j]} \right)^2, 
\end{align*}
where $\psi(\,\cdot\,): \BR \rightarrow \BR$ (defined in \Eqref{eq:hermitte-phi}) is a function with $\psi'(0) = \psi(0) = 0$. Define the sublevel set $ \mathcal{X} := \{ x: \varphi(x) \le \varphi(0)\}$. There exists numerical constants $c_1,c_2>0$ such that
\begin{enumerate}
\item $f(x,y)$ is $c_1$-Lipschitz in $y$ on $\fX \times \BR^{d_y}$;
    \item $f(x,y)$ has $c_2$-Lipschitz gradients on $\fX \times \BR^{d_y}$;
    \item $g(y)$ is a strictly convex quadratic and has $1$-Lipschitz gradients;
    \item The resulting hyper-objective is $\varphi(x) = (x+1)^2 /2$.
\end{enumerate}
For this problem,
any algorithm with a procedure as Algorithm \ref{alg:class}  stays at $x_t = 0$ for any iteration number $t \le T$.
\end{restatable}

In our construction, we let $\psi(t)$ be a function which is $ \frac{1}{2} t^2$ near zero, while remaining
bounded when $\vert t \vert$ is large. 
We assign different dimensions in $y$ to functions $\psi$ and $g$ separately.
Note that $g$ is designed such that any
zero-respecting algorithm can only make progress at most one dimension per oracle calls. As the progress in $g$ is slow, it will not affect the dimensions that $\psi$ depends on.

Below we discuss the insights brought by the hard instances that we have constructed.

\begin{rmk}
    Our results in these two subsections, from two complementary perspectives, motivate us to focus on more well-behaved lower-level functions that: (1) can confer the continuity of  $Y^*(x)$, to avoid the case as Example \ref{exmp:disconti}. (2) an algorithm can converge rapidly to a neighborhood of $Y^*(x)$, to avoid the case as Theorem \ref{thm:NO}. In the next section, we will show that the PL condition simultaneously satisfies both of these requirements.
\end{rmk}

\section{Positive Results for Lower-Level Functions Satisfying PL Conditions} \label{sec:YES}

As we have shown finding small hyper-gradients of nonconvex-convex bilevel problems is intractable for nonconvex-convex bilevel problems, we turn our attention to the tractable cases. 
Motivated by recent works~\citep{shen2023penalty,kwon2023penalty,arbel2022non}, we study the case when the lower-level problem satisfies the PL condition.

\subsection{ The Assumptions for Nonconvex-PL Bilevel Problems } \label{sec:intro-F2BA}

Without the typical lower-level strong convexity assumption, it is difficult to directly analyze the implicit gradient as \Eqref{grad-sc} since we can not directly use the implicit function theorem. Recently, \citet{kwon2023penalty} proposed a novel way to study the differentiability of $\varphi(x)$ for nonconvex-PL bilevel problems. Instead of directly studying the original hyper-objective $\varphi(x)$, they studied  the 
following value-function penalized hyper-objective {as a bridge}:
\begin{align} \label{eq:varphi-sigma}
   \varphi_{\sigma}(x):= \min_{y \in \fY} \left\{ f(x,y) + 
   \frac{g(x,y) - g^*(x)}{\sigma}
   \right\},
\end{align}
where  $g^*(x) = \min_{y \in \fY } g(x,y)$ is the lower-level value-function and $\fY \subseteq \BR^{d_y}$. They first studied the differentiability of $\varphi_{\sigma}(x)$ and then showed the limit of $\sigma \rightarrow 0^+$ exists.

However, the PL condition on $g$ is not sufficient to guarantee the differentiability of $\varphi_{\sigma}(x)$. Below, we give a concrete example to illustrate this.

\begin{exmp}
Consider Problem (\ref{prob:main}) with $d_x = 1$, $d_y = 2$. Let the upper-level function $f(x,y) =xy_{[1]} $, and the lower-level function $g(x,y) = \frac{1}{2} y_{[2]}^2$. The lower-level function is $1$-PL in $y$. Consider the penalized hyper-objective $\varphi_{\sigma}(x)$ in \Eqref{eq:varphi-sigma}. For any $\sigma \ge 0$,
if taking minimum in $y$ over all the domain $y \in \BR^{2}$, $\varphi_{\sigma}(x)$ is not well-defined since $\varphi_{\sigma}(x) = -\infty$ for any $x \ne 0$; if taking minimum in $y$ over a compact domain such as $\fY = [0,1] \times [0,1] $, $\varphi_{\sigma}(x) = \min \{ x,0\}$ is not differentiable at the point $x=0$.
\end{exmp}

To guarantee the differentiability of $\varphi_{\sigma}(x)$, \citet{kwon2023penalty} assumed the PL condition not only holds for $g$, but also holds for the penalty function $h_{\sigma} = \sigma f+ g$ uniformly for all $\sigma$ near zero. In addition to this assumption, the authors also imposed other standard smoothness assumptions which are typically required in previous works. These assumptions, as stated in \citet{kwon2023penalty}, are formally presented below.


\begin{asm}
\label{asm:PL} Recall the bilevel problem defined in~\Eqref{prob:main}, where $f$ is the upper-level problem, $g$ is the lower-level problem. Let $h_{\sigma}= \sigma f + g$ be the penalty function. Suppose that 
\begin{enumerate}[label=(\alph*)]
\item The penalty function $h_{\sigma}(x,y)$ is $\mu$-PL in $y$ for any $0 \le \sigma \le \overline{\sigma}$;
    \item The upper-level function $f(x,y)$ is $C_f$-Lipschitz in $y$ and has $L_f$-Lipschitz gradients;
    \item The lower-level function $g(x,y)$ has $L_g$-Lipschitz gradients;
    \item The upper-level function $f(x,y)$ has $\rho_f$-Lipschitz Hessians in $y$, \textit{i.e.} $\nabla_{xy}^2 f$ and $\nabla_{yy}^2 f$ are $\rho_f$-Lipschitz continuous;
\item The lower-level function $g(x,y)$ has $\rho_g$-Lipschitz Hessians.
\end{enumerate}
Under this assumption, we define the largest smoothness constant $\ell = \max\{C_f,L_f, L_g, \rho_g \}$ and the condition number $\kappa: = \ell/ \mu$.
\end{asm}

\begin{rmk} \label{rmk:asm}
Because we focus on the behavior when $\sigma$ is close to zero, whenever we mention $\sigma$ in the context, we always assume that $\sigma \in [0,\overline{\sigma}] $ even when this condition is not explicitly stated.
Note that all the assumptions are the same as \citet{kwon2023penalty}, except Assumption \ref{asm:PL}(a) may seem different from the Prox-EB assumption (Assumption 1 in \citet{kwon2023penalty}). In Appendix \ref{apx:Prox-F2BA} we show these two assumptions can imply each other in the unconstrained case. Our narrative uses the PL condition and is therefore more convenient. As a by-product, we show the proximal operator in \citet{kwon2023penalty} is unnecessary, and the original F${}^2$BA~\citep{chen2023near} can also converge under PL conditions.
In Assumption \ref{asm:PL}(d), \citet{kwon2023penalty} simplily assumes $f(x,y)$ has Lipschitz Hessians, but we note that their analysis only requires $\nabla_{xy}^2 f$ and $\nabla_{yy}^2 f$ are Lipschitz continuous. We use this refined assumption because the $\Omega(\epsilon^{-2})$ lower bound function for finding an $\epsilon$-stationary point of $f(x)$ does not have Lipschitz continuous Hessians in $x$, so we also do not assume $\nabla_{xx}^2 f$ is Lipschitz continuous in our upper bounds.
\end{rmk}

\begin{rmk} \label{rmk:huang}
Assumption \ref{asm:PL}(a) can be also replaced by $g(x,y)$ satisfies $\mu$-PL condition with a unique minimizer $y^*(x)$ and $\nabla_{yy}^2 g(x,y^*(x))$ is non-singular as \citet{huang2023momentum,huang2024optimal}. The analysis would be almost the same, but it is an easier case as discussed in Section 3.2.1 \citep{kwon2023penalty}. In this paper, we focus on the more challenging case as in \citep{kwon2023penalty}.
\end{rmk}
Assumption~\ref{asm:PL} ensures that the solution set $Y^*$ is stable under perturbations of $\sigma$ and $x$. 

\begin{restatable}{lem}{lemsetLip} \label{lem:set-Lip}
Let $Y^*_\sigma:= \arg \min_{y \in \BR^{d_y}} h_{\sigma}(x,y)$ denote the set of minima for the penalty function $h_\sigma(x, y) = \sigma f(x,y) + g(x,y)$. Under Assumption \ref{asm:PL}, we have that
    \begin{align*}
        {\rm dist}( Y_{\sigma_1}^*(x_1), Y_{\sigma_2}^*(x_2)) \le \frac{C_f}{\mu} \vert \sigma_1 - \sigma_2 \vert + \frac{\sigma L_f + L_g}{\mu} \Vert x_1 -x_2 \Vert.
    \end{align*}
\end{restatable} 

Then we can use Theorem \ref{thm:cont-cont} to get the continuity of hyper-objective $\varphi(x)$.


\subsection{F${}^2$BA Can also be Applied to Nonconvex-PL Bilevel Problems} \label{sec:alg-epsilon3}

Although F${}^2$BA is originally proposed for nonconvex-strongly-convex bilevel problems,
 we show that F${}^2$BA can also be applied to nonconvex-PL bilevel problems in this section. 
 
 Our starting points are the following lemmas that hold once we have Lipschitz continuity of solution set from Lemma \ref{lem:set-Lip}, which unnecessarily requires strong convexity. Firstly, we can obtain the following result by using the generalized Danskin's theorem  \citep{shen2023penalty} twice, specifically, in both $Y^*(x)$ and $Y_{\sigma}^*(x)$.

\begin{restatable}[\citet{shen2023penalty}]{lem}{lemDanskin} \label{lem:Danskin} 
Recall that $\varphi_{\sigma}(x)$ is the penalized hyper-objective defined in \Eqref{eq:varphi-sigma}.
Under Assumption \ref{asm:PL}, $\nabla \varphi_{\sigma}(x)$ exists and takes the form of 
    \begin{align} \label{eq:nabla-sigma}
    \nabla \varphi_{\sigma}(x) =  \nabla_x f(x,y_{\sigma}^*(x)) + \frac{\nabla_x g(x,y_{\sigma}^*(x)) - \nabla_x g(x,y^*(x))}{\sigma},
\end{align}
where $y^*(x)$, $y_{\sigma}^*(x)$ can be arbitrary elements in $Y^*(x)$ and $Y_{\sigma}^*(x)$, respectively.
\end{restatable}
Secondly, the stability of $Y_{\sigma}^*$ under perturbations of $\sigma$ by Lemma \ref{lem:set-Lip} implies the stability of $\varphi_{\sigma}$ and $\nabla \varphi_{\sigma}$ by invoking the result by
\citet{kwon2023penalty}. 

\begin{restatable}[\citet{kwon2023penalty}]{lem}{lemKwon} \label{lem:Kwon}
Recall that $\varphi(x)$ is the original hyper-objective in \Eqref{hyper-refo}, while $\varphi_{\sigma}(x)$ is the penalized hyper-objective in \Eqref{eq:varphi-sigma}.
Under Assumption \ref{asm:PL}, 
{$\nabla \varphi(x)$ exists and can be defined as the limit  $\lim_{\sigma \rightarrow 0^+} \nabla \varphi_{\sigma}(x)$ }. Moreover, $\varphi_{\sigma}(x)$ is close to  $\varphi(x)$. Formally, for any $0 \le \sigma \le \min\{\rho_g /\rho_f, \overline{\sigma} \}$, we have that
\begin{align*}
    \vert \varphi_{\sigma}(x) - \varphi(x) \vert &= \fO(\sigma \ell \kappa), \quad {\rm and} \quad 
    \Vert \nabla \varphi_{\sigma}(x) - \nabla \varphi(x) \Vert = \fO(\sigma \ell \kappa^3).
\end{align*}
\end{restatable}

These two lemmas
make it reasonable to apply
the F${}^2$BA (Algorithm \ref{alg:F2BA}), which applies gradient descent on $\varphi_{\sigma}(x)$ according to \Eqref{eq:nabla-sigma}.
Lemma \ref{lem:Kwon} shows that to find an $\fO(\epsilon)$-stationary point of $\varphi(x)$, it suffices to finds an $\fO(\epsilon)$-stationary point of $\varphi_{\sigma}(x)$ for $\sigma =\fO(\epsilon)$. 
Note that gradient descent can find an $\epsilon$-stationary point of a nonconvex $L$-smooth function within $\fO(L \epsilon^{-2})$ complexity and 
 that $\varphi_{\sigma}(x)$ is $\fO(\sigma^{-1})$-smooth in the worst-case.
Such an analysis illustrates an $\tilde \fO(\sigma^{-1} \epsilon^{-2}) = \tilde \fO(\epsilon^{-3})$ complexity as \citet{kwon2023penalty}.

\begin{algorithm*}[t]   
\caption{F${}^2$BA $\left(x_0,y_0, \eta,\tau, \sigma, T,K\right)$} \label{alg:F2BA}
\begin{algorithmic}[1] 
\STATE $ z_0 = y_0$ \\[1mm]
\STATE \textbf{for} $ t =0,1,\cdots,T-1 $ \\[1mm]
\STATE \quad $ y_t^0 =  y_{t}, ~ z_t^0 = z_{t}$ \\[1mm]
\STATE \quad \textbf{for} $ k =0,1,\cdots,K-1$ \\[1mm]
\STATE \quad \quad $ z_t^{k+1} = z_t^{k}- \tau \nabla_y g(x_t, z_t^k)   $ \\[1mm]
\STATE \quad \quad $ y_t^{k+1} = y_t^k - \tau \left( \sigma \nabla_y f(x_t,y_t^k) +   \nabla_y g(x_t,y_t^k) \right)$ \\[1mm]
\STATE \quad \textbf{end for} \\[1mm]
\STATE \quad $ \hat \nabla \varphi(x_t)= \nabla_x f(x_t,y_{t}^K) +  ( \nabla_x g(x_t,y_{t}^K) - \nabla_x g(x_t,z_{t}^K) ) / \sigma$ \\[1mm]
\STATE \quad  $x_{t+1} = x_t -  \eta  \hat \nabla \varphi(x_t)$ \\[1mm]
\STATE \textbf{end for} \\[1mm]
\end{algorithmic}
\end{algorithm*}

\subsection{Achieving the Near-Optimal Rate for the Deterministic Case} \label{sec:det}

Interestingly, we can show a rate of $\tilde \fO(\epsilon^{-2})$
of Algorithm \ref{alg:F2BA}.
Our improvement comes from a similar technique by \citet{chen2023near}, which shows that
$\varphi(x)$ is $\fO(1)$-smooth for nonconvex-strongly-convex problems. The intuition is that we can restrict our analysis to the strongly convex subspace induced by the PL condition, and apply the result of \citet{chen2023near}.


\begin{restatable}{lem}{lemgradLip} \label{lem:grad-Lip}
Under  Assumption \ref{asm:PL},
$\varphi(x)$ has $\fO(\ell \kappa^3)$-Lipschitz gradients.
\end{restatable}
We sketch the proof as follows. 
First, we give the explicit form of $\nabla \varphi(x)$ in Lemma \ref{lem:nabla-phi-form}. Next, we show that the PL condition ensures a strongly convex subspace near any minimum in Lemma \ref{lem:PL-singular}.
Finally,  We apply 
Lemma \ref{lem:space} and \ref{lem:UV} to project the functions onto the strongly convex subspace, where $\varphi(x)$ has $\fO(1)$-Lipschitz gradients to complete the proof.

Based on this lemma, we can readily prove our main theorem in this subsection, stated below.

\begin{restatable}{thm}{ThmYES} \label{thm:YES}
Suppose Assumption \ref{asm:PL} holds. Define
$\Delta:=\varphi(x_0) - \inf_{x \in \BR^{d_x}} \varphi(x)$, $R := {\rm dist}^2 (y_0,Y^*(x)) $ and supposed both $\Delta$, $R$ are bounded.
Set the parameters in Algorithm \ref{alg:F2BA} as
\begin{align*}
    &\eta \asymp \ell^{-1}\kappa^{-3},  \quad \sigma \asymp \min \left\{
    \frac{R}{\kappa}, ~  \frac{\epsilon}{\ell \kappa^3}, ~\frac{L_g}{L_f}, ~\frac{\rho_g}{\rho_f}, ~\overline{\sigma} \right\}, \\
    &~~~~~~~~~~\tau = \frac{1}{\sigma L_f+  L_g}, \quad K \asymp \frac{L_g }{\mu} \log \left( \frac{L_g}{\mu \sigma} \right),
\end{align*}
then it
can find an $\epsilon$-first-order stationary point of $\varphi(x)$ within $T = \fO(\ell \kappa^3 \epsilon^{-2})$ iterations, where $\ell$,  $\kappa$ are defined in Assumption~\ref{asm:PL}. The total number of first-order oracle calls is bounded by $\fO( \ell \kappa^4 \epsilon^{-2} \log(\nicefrac{\rho_f \ell \kappa }{\epsilon})) $.
\end{restatable}

The above complexity matches the optimal rate gradient descent for single-level nonconvex minimization by \citet{carmon2021lower}, except for an additional logarithmic factor.  


\begin{rmk}
One small difference between the assumptions in Theorem \ref{thm:YES} and those of the constructed hard instance in Theorem \ref{thm:NO}.
is that the hard instance only satisfies the Lipschitz-type conditions in $\fX \times \BR^{d_y}$, where $\fX$ is the sublevel set of $x_0$. However, indeed, the upper bound also only requires the Lipschitz-type conditions in the sublevel set by the descent lemma $\varphi(x_{t+1}) \le \varphi(x_t)$. 


\end{rmk}

\subsection{The Extensions of  F${}^2$BA to the Stochastic Case} \label{sec:stoc}

\begin{algorithm*}[t]  
\caption{F${}^2$BSA $\left(x_0,y_0, \delta_0, \eta,\tau, \sigma, T, B\right)$} \label{alg:stoc-F2BA}
\begin{algorithmic}[1] 
\STATE $ z_0 = y_0$ \\[1mm]
\STATE \textbf{for} $ t =0,1,\cdots,T-1 $ \\[1mm]
\STATE \quad $ y_t^0 =  y_{t}, ~ z_t^0 = z_{t}$ \\[1mm]
\STATE \quad Set $K_t$ as \Eqref{eq:para-stoc-F2BA} based on the value of $\delta_t$ \\[1mm]
\STATE \quad \textbf{for} $ k =0,1,\cdots,K_t-1$ \\[1mm]
\STATE \quad \quad $ z_t^{k+1} = z_t^{k}- \tau \nabla_y g(x_t, z_t^k; B)   $ \\[1mm]
\STATE \quad \quad $ y_t^{k+1} = y_t^k - \tau \left( \sigma \nabla_y f(x_t,y_t^k; B) +   \nabla_y g(x_t,y_t^k;B) \right)$ \\[1mm]
\STATE \quad \textbf{end for} \\[1mm]
\STATE \quad $ \hat \nabla \varphi (x_t)= \nabla_x f(x_t,y_{t}^K; B) +  ( \nabla_x g(x_t,y_{t}^K; B) - \nabla_x g(x_t,z_{t}^K; B) ) / \sigma$ \\[1mm]
\STATE \quad  $x_{t+1} = x_t -  \eta  \hat \nabla \varphi (x_t)$ \\[1mm]
\STATE \quad Calculate $\delta_{t+1}$ as \Eqref{eq:delta_t} based on $\delta_t$, $x_{t+1}$ and $x_t$ \\[1mm]
\STATE \textbf{end for} \\[1mm]
\end{algorithmic}
\end{algorithm*}

Our analysis can also lead to better bounds on the stochastic case, where $f(x,y),g(x,y)$ are both the expectation of some stochastic components $F(x,y;\xi)$ and $G(x,y;\zeta)$, indexed by random vectors $\xi$ and $\zeta$:
\begin{align*}
    f(x,y) := \BE_{\xi} \left[ F(x,y ; \xi) \right], \quad g(x,y) := \BE_{\zeta} \left[ G(x,y ; \zeta) \right].
\end{align*}
An algorithm has access to the stochastic gradients, with the following assumptions.
\begin{asm} \label{asm:stochastic}We study the stochastic problem under the two assumptions below.
\begin{enumerate}[label=(\alph*)] 
    \item Suppose the stochastic gradients  are unbiased: 
    \begin{align*}
        \BE_{\xi} \left[ \nabla F(x,y;\xi) \right] = \nabla f(x,y), \quad \BE_{\zeta} \left[ \nabla G(x,y;\zeta) \right] = \nabla g(x,y);
    \end{align*}
    \item Suppose the stochastic gradients  have bounded variance. In other words,  there exist some constants $M_f, M_g>0$ such that
\begin{align*}
    \BE_{\xi} \left[ \Vert \nabla F(x,y;\xi) - \nabla f(x,y) \Vert^2 \right] \le M_f^2, \quad \BE_{\zeta} \left[ \Vert \nabla G(x,y;\zeta) - \nabla g(x,y) \Vert^2 \right] \le M_g^2.
\end{align*}
\end{enumerate}
\end{asm}


Under these assumptions, the natural extension of Algorithm \ref{alg:F2BA} to the stochastic setting is to replace the full-batch gradient in Algorithm \ref{alg:F2BA} to the mini-batch gradient, defined as follows.

\begin{dfn}
Given mini-batch size $B$. We define the gradient estimators for the upper-level and lower-level functions using the notations below,
\begin{align*}
    \nabla f(x,y; B) = \frac{1}{B} \sum_{i=1}^B \nabla f(x,y;\xi_i), \quad 
    \nabla g(x,y;B) = \frac{1}{B} \sum_{i=1}^B \nabla f(x,y;\zeta_i),
\end{align*}
where both $\xi_i$ and $\zeta_i$ are sampled i.i.d. 
\end{dfn}

 By replacing all the full-batch gradients in deterministic Algorithm~\ref{alg:F2BA}, we get the stochastic counterpart Algorithm \ref{alg:stoc-F2BA}, namely F${}^2$BSA (Fully First-order Bilevel Stochastic  Approximation). 
By additionally taking into account the error from stochastic gradients, we can extend Theorem \ref{thm:YES} to also tackle the stochastic case, yielding the following theorem.

\begin{restatable}{thm}{ThmStoc} \label{thm:stoc-F2BA}
Suppose Assumption \ref{asm:PL} and \ref{asm:stochastic} hold.  Define
$\Delta:=\varphi(x_0) - \inf_{x \in \BR^{d_x}} \varphi(x)$, $R := {\rm dist}^2 (y_0,Y^*(x)) $ and supposed both $\Delta$, $R$ are bounded.
Set the parameters in Algorithm \ref{alg:stoc-F2BA} as
\begin{align} \label{eq:para-stoc-F2BA}
\begin{split}
    & ~~~~~~~~~~~~~~~\eta \asymp \ell^{-1}\kappa^{-3}, \quad \sigma \asymp \min \left\{ \frac{R}{\kappa}, ~  \frac{\epsilon}{\ell \kappa^3}, ~ \frac{L_g}{L_f}, ~\frac{\rho_g}{\rho_f},~ \overline{\sigma} \right\}, \\
&      \tau = \frac{1}{\sigma L_f+  L_g},\quad {B \asymp \frac{L_g^2 \left(\sigma^2 M_f^2 + M_g^2 \right)}{\mu^2 \sigma^2 \epsilon^2} }  , \quad K_t \asymp \frac{L_g }{\mu} \log \left( \frac{L_g^3 \delta_t}{\mu \sigma^2 \epsilon^2} \right),
\end{split}
\end{align}
where $\delta_t$ is defined via the recursion
\begin{align} \label{eq:delta_t}
    \delta_{t+1} 
    &\le \frac{1}{2} \delta_t +\frac{8 L_g^2}{\mu^2} \Vert x_{t+1} - x_t \Vert^2 + \fO \left(\frac{\sigma^2 \epsilon^2}{L_g^2} \right), \quad \delta_0 \asymp R.
\end{align}
Then Algorithm \ref{alg:stoc-F2BA}
can find an $\epsilon$-first-order stationary point of $\varphi(x)$ in expectation within $T = \fO(\ell \kappa^3 \epsilon^{-2})$ iterations, where $\ell$,  $\kappa$ are defined in Definition \ref{asm:PL}. The total number of stochastic first-order oracle calls is bounded by
\begin{align*}
\tilde \fO( \kappa T B) = 
\begin{cases}
    \fO(\ell \kappa^4 \epsilon^{-2} \log (\nicefrac{\rho_f \ell \kappa}{\epsilon})  ), &  M_f =0, M_g=0;  \\
    \fO(\ell \kappa^6 \epsilon^{-4} \log (\nicefrac{\rho_f \ell \kappa}{\epsilon})), & M_f >0 , M_g = 0;\\
    {\fO(\ell^3 \kappa^{12} \epsilon^{-6} \log (\nicefrac{\rho_f \ell \kappa}{\epsilon})) }, & M_f>0, M_g >0.
\end{cases}
\end{align*}
\end{restatable}


In the deterministic case $(M_f=0,M_g=0)$, this result
recovers the $\tilde \fO(\epsilon^{-2})$ rate by Algorithm \ref{alg:F2BA}. 
In the partially stochastic case ($ M_f>0, M_g=0$), the $\tilde \fO(\epsilon^{-4})$ complexity is also near-optimal~\citep{arjevani2023lower}. In the fully stochastic case ($M_f>0, M_g>0$), our $\tilde \fO(\epsilon^{-6})$ upper bound is also better than the $\tilde \fO(\epsilon^{-7})$ upper bound by \citet{kwon2023penalty}. 



\section{Conclusions and Future Works}

This paper investigates bilevel optimization without the typical lower-level strong convexity assumption.
We have shown that finding points with small hyper-gradients is computationally hard for nonconvex-convex bilevel problems, but easy for nonconvex-PL bilevel problems where simple first-order algorithms can achieve fast convergence rates.

It will be interesting to study bilevel problems with lower-level functions beyond the PL condition in the future. One possible direction is to consider the 
Kurdyka-Łojasiewicz (KL) condition~\citep{fatkhullin2022sharp}, which is more general than the PL condition and it holds for any semialgebraic functions.
Recent works have shown non-asymptotic convergence for nonconvex-KL minimax optimization problems~\citep{li2022nonsmooth,zheng2022doubly}, but the nonconvex-KL bilevel optimization problem remains challenging. Besides these conditions, we hope future works to exploit the problem structure that arises from the applications of bilevel optimization and identify other tractable conditions that can break the limit of our hardness result.
\bibliography{sample}

\clearpage
\appendix

\section{Related Works} \label{apx:related}


Many existing works for non-asymptotic analysis for bilevel optimization assume the lower-level function is strongly convex. A natural method is to estimate $\nabla y^*(x)$, and plug into \Eqref{grad-sc} to get an estimator of $\nabla \varphi(x)$ to apply gradient descent.
ITD (ITerative Differentiation) \citep{gould2016differentiating, franceschi2017forward,shaban2019truncated,bolte2021nonsmooth} approximates $\nabla y^*(x)$ by ${\partial y^K(x)}  / {\partial x}$, where $y^K(x)$ is $K$-steps of gradient descent. 
AID (Approximate Implicit Differentiation) \citep{domke2012generic,ghadimi2018approximation,pedregosa2016hyperparameter,franceschi2018bilevel,grazzi2020iteration,ji2021bilevel,dagreou2022framework,arbel2021amortized} explicitly solve $\nabla y^*(x) =  \left[ ( \nabla_{yy}^2 g)^{-1} \nabla_{yx}^2 g \right] (x,y^*(x))$ as a linear equation. However, both AID and ITD require Hessian-vector product oracles, and their convergence analysis is typically restricted to the case when the lower-level function is strongly convex.
\citet{arbel2022non} extended implicit differentiation to the parametric Morse-Bott function, a class of nonconvex functions with local PL properties. They also proposed an 
algorithm that combines AID and ITD and showed its limit points must be an equilibrium of the concept of BGS (Bilevel Game with Selection) that they introduced.
\citet{yang2024achieving} proposed a first-order algorithm FdeHBO by estimating Jacobian/Hessian vector-product in AID with gradient differences. 

Recently, F${}^2$BA (Fully First-order Bilevel Approximation) has gained widespread attention. This method uses the value-function-based reformulation~\citep{ye1997exact} of Problem (\ref{prob:main}):
\begin{align} \label{value-refo}
    \min_{x \in \BR^{d_x}, y \in \BR^{d_y}} f(x,y), \quad {\rm s.t.} \quad g(x,y) = g^*(x)
\end{align}
and 
applies gradient descent on the penalty function. Several studies have demonstrated two notable advantages of this method over AID and ITD: Firstly, it only requires gradient oracles which are much cheaper than Hessian-vector product oracles~\citep{liu2022bome}. Secondly, it performs well empirically even in cases where the lower-level function is nonconvex~\citep{liu2021value,shen2023penalty}. By establishing the equivalence between the KKT (Karush–Kuhn–Tucker) point of Problem (\ref{value-refo}) and the stationary point of Problem (\ref{hyper-refo}), \citet{kwon2023fully} showed that F${}^2$BA can find an $\epsilon$-stationary point of $\varphi(x)$ with $\tilde \fO(\epsilon^{-3})$ first-order oracle calls 
when the lower-level function is strongly convex. 
\citet{chen2023near} further improved the rate to
$\tilde \fO(\epsilon^{-2})$ by a more careful landscape analysis of the penalty function.
\citet{kwon2023penalty,yao2024constrained} proposed proximal variants of F${}^2$BA by replacing the lower-level value function with its Moreau envelope to tackle lower-level constraints. \citet{kwon2023penalty} 
showed 
a $\tilde \fO(\epsilon^{-3})$ complexity when the penalty function satisfies the Prox-EB condition (which is equivalent to the PL condition in the unconstrained case as Proposition~\ref{prop:Prox-EB-PL}).

GALET (Generalized
ALternating mEthod for bilevel opTimization) by \citet{xiao2023generalized} can also tackle nonconvex-PL bilevel optimization problems. It uses Hessian-vector-product oracles to solve the following gradient-based reformulation of Problem (\ref{prob:main}):
\begin{align} \label{grad-refo}
    \min_{x \in \BR^{d_x}, y \in \BR^{d_y}} f(x,y), \quad {\rm s.t.} \quad \nabla_y g(x,y) = 0.
\end{align}
\citet{xiao2023generalized} showed GALET can converge to a KKT point of Problem (\ref{grad-refo}) with a rate of $\tilde \fO(\epsilon^{-2})$.
In Appendix \ref{apx:GALET}, we show that this also implies the same convergence rate to
the stationary point of $\varphi(x)$ under the assumptions of \citet{kwon2023penalty}.

\citet{huang2023momentum,huang2024optimal} also considers nonconvex-PL bilevel problems.
Our work differs from these works in the following aspects. First, Assumption 2 in \citep{huang2023momentum} $g(x,y)$ has a unique minimizer $y^*(x)$ and $\nabla_{yy}^2 g(x,y^*(x))$ is non-singular, which is much easier than our setting as we have discussed in Remark \ref{rmk:huang}. 
Secondly, our algorithm only requires first-order oracles, while \citep{huang2023momentum} requires second-order oracles. Thirdly, our algorithm has only $\fO(d)$ complexity at each step, where $d = \max\{d_x,d_y \}$ is the dimension of the problem. In contrast, each iteration of the algorithm in \citep{huang2023momentum} requires a $\fO(d^3)$ complexity for computing the SVD decomposition of the Hessian matrix.
Although \citet{huang2024optimal} claimed a projection operator can remove the expensive SVD decomposition, the claim seems to be incorrect. \citet{huang2024optimal} defined projector $\fM(H v )$ such that $ \fM(H v ) = \fS_{[\mu,L_g]} (H) \fP_{r}(v)$, where $\fS_{[\mu,L_g]} (H)$ projects all the singular-values of $H$ into the interval $[\mu,L_g]$ and $\fP_{r}(v)$ projects vector $v$ onto the set $\{v \in \BR^{d_y} : \Vert v \Vert \le r \}$. Consider $H =\begin{pmatrix}
2 & 0 \\
0 & 1 
\end{pmatrix}$, $v = 
\begin{pmatrix}
1 \\
1 
\end{pmatrix}
$, $L= \frac{3}{2}$, $\mu=\frac{1}{2}$, $r_v = \sqrt{2}$ then $Hv =  \begin{pmatrix}
2 \\
1 
\end{pmatrix}$, but $\fM(Hv) = 
\begin{pmatrix}
\frac{3}{2} \\
1 
\end{pmatrix}$ by definition. Hence, $\fM(Hv) \ne \fP_{r'}(Hv)$ for any $r'$. 
In fact, it seems that no deterministic function can implement $\fM(Hv)$. Let $v' = \begin{pmatrix}
2 \\
1 
\end{pmatrix} $, $H' = \begin{pmatrix}
1 & 0 \\
0 & 1
\end{pmatrix}$. Then $\fM(H'v') = \frac{\sqrt{10}}{5} 
\begin{pmatrix}
2 \\
1 
\end{pmatrix}
$ by definition. We have $\fM(H'v') \ne \fM(Hv)$, but $H'v' = Hv$.
Therefore, it seems that the algorithms in \citep{huang2023momentum,huang2024optimal} actually both require very expensive Hessian oracles with $\fO(d^3)$ running time each iteration, which is different from previous works (and our work) that only use Hessian-vector product oracles or gradient oracles with $\fO(d)$ running time per iteration.

Some works use a sequential
approximation minimization strategy~\citep{liu2021towards,liu2021value,liu2020generic} to tackle the discontinuous challenge for bilevel problems without lower-level strong convexity as illustrated by Example \ref{exmp:disconti}. These works generate a series of continuous functions $\{ \varphi_K\}_{K=0}^{\infty}$ which converge to $\varphi$ when $K \rightarrow \infty$.
\citet{liu2020generic} proposed BDA by defining $\varphi_K(x) = f(x,{\rm AggrGD}_{f,g}^K(x))$ for nonconvex-convex problems, where ${\rm AggrGD}_{f,g}^K(x)$ denotes $K$-steps of 
gradient descent on the aggregated function $ 
\alpha f(x,\,\cdot\,)+ (1-\alpha) g(x,\,\cdot\,)$ for some $\alpha \in (0,1)$. \citet{liu2021towards} defined $\varphi_K(x) = \min_z \max_{k \le K} \{ f(x, {\rm GD}^k(x;z))\}$ for nonconvex-nonconvex problems, where ${\rm GD}^k(x;z))$ denotes $k$-steps of gradient descent on $g(x,\,\cdot\,)$ initialized with $z$.
However, all these works require solving a series of sub-problems and often lack a non-asymptotic analysis.

\section{Discussion on \citet{xiao2023generalized}} \label{apx:GALET}

The recently proposed GALET~\citep{xiao2023generalized} has been shown to converge with the rate $\tilde \fO(\epsilon^{-2})$ to the following definition of stationary points of Problem (\ref{grad-refo}).

\begin{dfn} \label{dfn:sta-GALET}
    We say $x^*$ is an $\epsilon$-stationary point of Problem (\ref{grad-refo}) if $\exists y^*, w^*$ such that
    \begin{align*}
        \Vert \fR_x(x^*,y^*,w^*) \Vert \le \epsilon, \quad \Vert \fR_w(x^*,y^*,w^*) \Vert \le \epsilon, \quad \fR_y(x^*,y^*)  \le \epsilon^2,
    \end{align*}
    where we define
    \begin{align*}
        \fR_x(x,y,w) &= \nabla_x f(x,y) + \nabla_{xy}^2 g(x,y) w; \\
        \fR_w(x,y,w) &= \nabla_{yy}^2 g(x,y)\left( \nabla_y f(x,y) + \nabla_{yy}^2 g(x,y) w \right) ; \\
        \fR_y(x,y) &= g(x,y) - \min_{z \in \BR^{d_y}} g(x,z).
    \end{align*}
\end{dfn}

We can show the $\epsilon$-stationary point of $\varphi(x)$ implies the above stationarity definition.

\begin{prop}
Under Assumption \ref{asm:PL}, if
$x^*$ is an $\epsilon$-stationary point of $\varphi(x)$ in \Eqref{hyper-refo}, then it is also an $\fO(\epsilon)$-stationary point as Definition \ref{dfn:sta-GALET}.
\end{prop}

\begin{proof}
Given $x^*$, take any $y^* \in Y^*(x)$, and let $w^* = -\left( \nabla_{yy}^2 g(x^*,y^*) \right)^{\dagger} \nabla_y  f(x^*,y^*) $. Then we have both $\fR_y(x^*,y^*) = 0$ and $ \fR_{w}(x^*,y^*,w^*) = 0$. By Lemma \ref{lem:nabla-phi-form}, we know that $\fR_x(x^*,y^*,w^*) = \nabla \varphi(x^*)$, therefore we also have $ \Vert \fR_x(x^*,y^*,w^*) \Vert \le \epsilon$.
\end{proof}
%

Furthermore, the converse relationship also holds.


\begin{prop}
    Under Assumption \ref{asm:PL} , if $\hat x$ is an $\epsilon$-stationary point as Definition \ref{dfn:sta-GALET}, then it is also an $\fO(\epsilon)$-stationary point of $\varphi(x)$.
\end{prop}

\begin{proof}
Let $y^*,w^*$ be the corresponding point of $x^*$ in Definition \ref{dfn:sta-GALET}. By Lemma \ref{lem:PL-QG}, 
    \begin{align*}
        {\rm dist} \left( y^*, Y^*(x) \right) \le \sqrt{\frac{\mu}{2} \fR_y(x^*,y^*)} \le \sqrt{\frac{\mu}{2}}\cdot \epsilon.
    \end{align*}
    
Take $\hat y = \arg \min_{y \in Y^*(x)} \Vert y^* - y \Vert$ and
define the following auxiliary function:
\begin{align} \label{eq:square}
    \fL(w):= \frac{1}{2} \left \Vert \nabla_{yy}^2 g(x^*,\hat y) w + \nabla_y f(x^*,\hat y) \right \Vert^2.
\end{align}
Note that 
\begin{align*}
    \nabla \fL(w) = \nabla_{yy}^2 g(x^*,\hat y)\left( \nabla_y f(x^*,\hat y) + \nabla_{yy}^2 g(x^*,\hat y) w \right).
\end{align*}
When $\Vert w^* \Vert =\fO(1)$, we have $ \Vert \nabla \fL(w^*) - \fR_w(x^*,y^*,w^*) \Vert =\fO(\epsilon)$.

$\fL(w)$ has the form of "strongly convex composed with linear". It is $s^2$-PL by \citet{karimi2016linear}, Appendix B, where $s$ is the smaller singular value of $\nabla^2 g(x^*,\hat y)$. By Lemma \ref{lem:PL-singular}, $s \ge \mu$.

Let $W^* = \arg \min_{w \in \BR^{d_y}} \fL(w)$. Then by Lemma \ref{lem:PL-EB},
\begin{align*}
    {\rm dist} \left(w^*, W^* \right) \le \frac{1}{s^2} \Vert \nabla \fL(w^*) \Vert = \fO(\epsilon).
\end{align*}
Note that $W^*$ has the explicit form of 
\begin{align*}
    W^* = -\left(\nabla_{yy}^2 g(x^*,\hat y)  \right)^\dagger  \nabla_y f(x^*,\hat y) + {\rm Ker}(\nabla_{yy}^2 g(x^*,\hat y) ).
\end{align*}
Take $\hat w = \arg \min_{w \in W^*} \Vert w-w^* \Vert$ and by Lemma \ref{lem:space}, 
\begin{align*}
    \nabla \varphi(x^*) = \fR_x(x^*,\hat y,\hat w) = \nabla_x f(x^*,\hat y) + \nabla_{xy}^2 g(x^*,\hat y) \hat w.
\end{align*}
Since $w^*$ is bounded by Lemma \ref{lem:PL-singular}, we know that 
\begin{align*}
    \Vert \nabla \varphi(x^*) \Vert \le \Vert \fR_x(x^*,y^*,w^*) \Vert +  \Vert \fR_x(x^*,\hat y,\hat w) - \fR_x(x^*,y^*,w^*) \Vert = \fO(\epsilon).
\end{align*}


\end{proof}

\begin{rmk} \label{rmk:comp-GALET}
    In above analysis, the optimal $w^*$ of GALET~\citep{xiao2023generalized} is the solution to the linear equation $ \nabla_{yy}^2 g(x,y^*) w = \nabla_y g(x,y^*)$. However, the Hessian matrix $\nabla_{yy}^2 g(x,y)$ may be indefinite when the lower-level function is nonconvex. To overcome this issue, \citet{xiao2023generalized} uses the square trick like \citet{liu2022quasi}, which solves \Eqref{eq:square} instead. One drawback of the square trick is the conditional number becomes $\fO(\kappa^2)$ and makes the inner loop slower. Since $\varphi(x)$ has $\fO(\ell \kappa^3)$-Lipschitz gradients by Lemma \ref{lem:grad-Lip}. According to our above analysis, the reasonable complexity of GALET should be $\tilde \fO(\ell \kappa^5 \epsilon^{-2})$  though the dependency on $\kappa$ is not explicitly given by \citet{xiao2023generalized}. 
\end{rmk}


\section{Discussion on \citet{kwon2023penalty}} \label{apx:Prox-F2BA}

Different from our Assumption \ref{asm:PL}a, \citet{kwon2023penalty} uses the following proximal error bound  (Prox-EB) assumption.
\begin{asm} \label{asm:ProxEB}
    Let $ h_{\sigma}(x,y) := \sigma f(x,y) + g(x,y) $ and $Y_{\sigma}^*(x) := \arg \min_{y \in \BR^{d_y}} h_{\sigma}(x,y)$. Suppose that for all $0 \le \sigma \le \overline{\sigma}$, there exists some $\mu'>0$ such that for all $y \in \BR^{d_y}$ we have 
    \begin{align*}
        \rho^{-1} \left \Vert y - y_{\sigma,\rho}^+(x) \right \Vert \ge \mu' {\rm dist} \left(y, Y_{\sigma}^*(x) \right),
    \end{align*}
    where $y_{\sigma,\rho}^+(x)$ is defined via the proximal operator with parameter $\rho$ for function $h_{\sigma}(x,\,\cdot\,)$ as
    \begin{align*}
    y_{\sigma,\rho}^+(x) :=     \arg \min_{z \in \BR^{d_y}} \left\{ 
        h_{\sigma}(x,z) + \frac{1}{2 \rho} \Vert y -z \Vert^2. 
         \right\}.
    \end{align*}
\end{asm}

Below, we show that this assumption is equivalent to our Assumption \ref{asm:PL}a.

\begin{prop} \label{prop:Prox-EB-PL}
Let $0 \le \sigma \le \min\{ L_g/L_f, \overline{\sigma} \}$. Suppose both c and d in Assumption \ref{asm:PL} hold. Let $\rho < 1/(2L_g)$.
If Assumption \ref{asm:PL}a holds with constant $\mu$, then Assumption \ref{asm:ProxEB} holds with $\mu' = \mu / (1+2L_g \rho)$. Conversely, if Assumption \ref{asm:ProxEB} holds with constant $\mu'$, then Assumption \ref{asm:PL}a holds with $\mu = \left(\mu' (1- 2L_g \rho) \right)^2 / (2L_g)$. 
\end{prop}

\begin{proof}
Let $\sigma \le L_g /L_f$ then $h_{\sigma}(x,y)$ has $(2L_g)$-Lipschitz gradients.
    Define the Moreau envelope for $h_{\sigma}(x,y)$ with respect to $y$ as $ h_{\sigma,\rho}(x,y) :=  h_{\sigma} \left( x,y_{\sigma,\rho}^+(x)\right)$. Let $\rho < 1/(2L_g)$ then $y_{\sigma,\rho}^+(x)$ is uniquely defined. Danskin's theorem implies that $ \nabla_y h_{\sigma,\rho}(x,y) = \rho^{-1} (y - y_{\sigma,\rho}^+(x) ) $. Therefore the Prox-EB assumption is equivalent to $\Vert \nabla_y h_{\sigma,\rho}(x,y) \Vert \ge \mu' {\rm dist} (y, Y_{\sigma}^*(x) )$.
    Note that
    \begin{align*}
       &\quad  \Vert \nabla_y h_{\sigma,\rho}(x,y) - \nabla_y h_{\sigma}(x,y) \Vert \\
       &= \Vert \nabla_y h_{\sigma}\left( x, y_{\sigma,\rho}^+(x)\right) - \nabla_y h_{\sigma}(x,y) \Vert \\ &\le 2L_g \Vert y - y_{\sigma,\rho}^+(x) \Vert \\
       &=2L_g \rho \Vert \nabla_y h_{\sigma,\rho}(x,y) \Vert.
    \end{align*}
When $\rho < 1/(2L_g)$ the triangle inequality implies
\begin{align*}
    (1 - 2 L_g \rho)  \Vert \nabla_y h_{\sigma,\rho}(x,y) \Vert \le \Vert \nabla_y h_{\sigma}(x,y) \Vert \le  (1 + 2 L_g \rho)  \Vert \nabla_y h_{\sigma,\rho}(x,y) \Vert .
\end{align*}
Therefore, if $\mu$-PL condition holds,  by Lemma \ref{lem:PL-EB}, we have
\begin{align*}
    \Vert \nabla_y h_{\sigma,\rho} (x,y) \Vert \ge \frac{1}{1+ 2L_g \rho} \Vert \nabla_y h_{\sigma} (x,y) \Vert \ge \frac{\mu}{1+ 2L_g \rho} {\rm dist} \left(y,Y_{\sigma}^*(x) \right).
\end{align*}
Conversely, if $\mu'$-Prox-EB condition holds,  then we have
\begin{align*}
    \Vert \nabla_y h_{\sigma}(x,y) \Vert \ge (1- 2L_g \rho) \Vert \nabla_y h_{\sigma,\rho}(x,y) \Vert \ge \mu'(1-2L_g\rho) {\rm dist} \left(y,Y_{\sigma}^*(x) \right).
\end{align*}
Since $h_{\sigma}(x,y)$ has $(2L_g)$-Lipschitz gradients, then 
\begin{align*}
    h_{\sigma}(x,y) - h^*_{\sigma}(x) \le L_g {\rm dist}^2 (y, Y_{\sigma}^*(x) ) \le \frac{L_g}{ \left( \mu' \right)^2 (1-2L_g \rho)^2} \Vert \nabla_y h_{\sigma}(x,y) \Vert^2, 
\end{align*}
which is exactly the PL inequality.

\end{proof}

\section{Proof of Theorem \ref{thm:cont-cont}}

\thmcontcont*
\begin{proof}
For given $x_1,x_2$, define 
\begin{align*}
    d_1 := \max_{y_2 \in Y^*(x_2)} {\rm dist}(Y^*(x_1),y_2), \quad d_2:= \max_{y_1 \in Y^*(x_1)} {\rm dist}( y_1, Y^*(x_2)).
\end{align*}
Note that we can replace sup with max in Definition \ref{def:hausdorff} due to the compactness of $Y^*(x)$.  
Therefore,  $ {\rm dist}(Y^*(x_1), Y^*(x_2)) = \max \{d_1,d_2 \}$. Below we prove each part of the theorem.

(a). See Theorem 3B.5 \citep{dontchev2009implicit}.

(b).
It suffices to show for any given $x_1 \in \BR^{d_x}$ and any $\epsilon>0$, there exists $\delta>0$ such that for any $x_2$ satisfying $\Vert x_1-x_2 \Vert \le \delta$ both $d_1$ and $d_2$  are no larger than $\epsilon$. We prove this by assigning different $f(x,y)$ and then applying the continuity of $\varphi(x) := \min_{y \in Y^*(x)} f(x,y)$.

Firstly, take $f(x,y)=- {\rm dist}(y,Y^*(x_1))$. Simple calculus shows $\varphi(x_1) =0$ and $\varphi(x_2) = -d_1$.
By the continuity of $\varphi(x)$ at $x_1$, we know that for given $\epsilon>0$,  there exists $\delta_1>0$, such that for any $x_2$ satisfying $\Vert x_1-x_2\Vert \le \delta_1$, we have $d_1 = \varphi(x_1) - \varphi(x_2)  \le \epsilon$.

Secondly, we want to prove that for any $\epsilon>0$ there exists $\delta_2>0$ such that for any $x_2$ satisfying $\Vert x_1 - x_2 \Vert \le \delta_2$ we have $d_2 \le \epsilon$. We prove this by contradiction. Suppose not, then we can find a sequence $\{ x_n\}$ such that $ x_n \rightarrow x_1$, but $ \max_{y_1 \in Y^*(x_1)} {\rm dist}(y_1,Y^*(x_n)) \ge \epsilon$ for some $\epsilon>0$. We take the corresponding $y_n = \arg \max_{y_1 \in Y^*(x_1)} {\rm dist}(y_1,Y^*(x_n))$. Since $\{y_n \}$ is a bounded sequence, there exists a convergent subsequence $\{ y_{n_k} \}$ with some limit point $y_1' \in Y^*(x_1)$.
Take $n^*$ sufficiently large such that for any $n \ge n^*$ we have $\Vert y_{n_k} - y_1' \Vert \le \epsilon/2$. Then by triangle inequality, for any $n \ge n^*$ we have ${\rm dist}(y_1',Y^*(x_{n_k})) \ge \epsilon / 2$. Now, take $ f(x,y) =  \Vert y - y_1' \Vert $. Simple calculus shows that $\varphi(x_1) = 0$ and $\varphi(x_{n_k}) = {\rm dist}(y_1', Y^*(x_{n_k}))$. However, $ \vert \varphi(x_{n_k}) - \varphi(x_1) \vert \ge \epsilon/2$ for the sequence $\{x_{n_k}:n_k \ge n^* \}$ satisfying $x_{n_k} \rightarrow x_1$. This contradicts the continuity of $\varphi(x)$ at $x_1$.

Finally, we take $\delta = \min\{\delta_1,\delta_2 \}$ and conclude that once $\Vert x_1 - x_2 \Vert \le \delta$ we have both $d_1$ and $d_2$ are smaller than $\epsilon$, implying the continuity of $Y^*(x)$.

(c). It suffices to show for any given $x \in \BR^{d_x}$, there exists $\delta>0$ and $L>0$ such that $\varphi(x)$ is $L$-Lipschitz on $\sB_{\delta}(x_1)$. 
Firstly, the local Lipschitz continuity of $Y^*(\,\cdot\,)$ implies the existence of $\delta >0$ and $L_1>0$ such that $Y^*(\,\cdot\,)$ is $L_1$-Lipschitz on $\sB_{\delta}(x_1)$. Next, for any $x_2 \in \sB_{\delta}(x_1)$, we pick 
\begin{align} \label{eq:y1-y2}
    y_1\in\argmin_{y\in Y^*(x_1)} f(x_1,y), \quad y_2\in\argmin_{y\in Y^*(x_2)} f(x_2,y).
\end{align}
There exist $y_1' \in Y^*(x_1)$ and $y_2' \in Y^*(x_2)$ such that 
\begin{align*}
    \Vert y_1' - y_2 \Vert \le L_1 \Vert x_1 -x_2 \Vert,\quad \Vert y_1 - y_2' \Vert \le L_1 \Vert x_1 - x_2 \Vert.
\end{align*}
Therefore, both $y_2,y_2'$ lie in the compact set $\fN_y(x_1) = \{y: {\rm dist}(y,Y^*(x_1)) \le L_1 \delta \}$.   

The local Lipschitz property of $f(x,y)$ implies that there exists $L_2>0$ such that $ f(x,y)$ is $L_2$-Lipschitz on the set $\sB_{\delta}(x_1) \times \fN_y(x_1)$. Then
{\small \begin{align*}
    &\varphi(x_1)-\varphi(x_2)\le f(x_1,y_1')-f(x_2,y_2)\le L_2\left(\|x_1-x_2\|+\|y_2-y_1^\prime\|\right)\le (L_1+1) L_2 \Vert x_1  -x_2 \Vert.\\
    &\varphi(x_2)-\varphi(x_1)\le f(x_2,y_2')-f(x_1,y_1)\le L_2\left(\|x_1-x_2\|+\|y_1-y_2^\prime\|\right)\le (L_1+1) L_2 \Vert x_1  -x_2 \Vert.
\end{align*}}
This implies that $\varphi(x)$ is Lipschitz on $\sB_{\delta}(x_1)$.

(d).
For any compact set $K \subseteq \BR^{d_x}$, there exists $L>0$ such that $\varphi(x)$ is $L$-Lipschitz on $K$. Then
for any $x_1,x_2 \in K$, taking $f(x,y) = - {\rm dist}(y,Y^*(x_1))$ yields
\begin{align*}
    d_1 =\varphi(x_1) -  \varphi(x_2) \le L \Vert x_1 - x_2 \Vert, 
\end{align*}
By symmetric, we can also show that $d_2 \le L \Vert x_1 - x_2 \Vert$. Combining them, we show that $Y^*(x)$ is Lipschitz on any compact set $K$. This finishes the proof.

(e). Pick $y_1,y_2$ as \Eqref{eq:y1-y2}. Similarly, there exist $y_1^\prime\in Y^*(x_1)$ and $y_2^\prime\in Y^*(x_2)$ such that
\begin{align*}
    &\varphi(x_1)-\varphi(x_2)\le f(x_1,y_1^\prime)-f(x_2,y_2)\le C_f\left(\|x_1-x_2\|+\|y_2-y_1^\prime\|\right)\le (\kappa+1) C_f \Vert x_1  -x_2 \Vert,\\
    &\varphi(x_2)-\varphi(x_1)\le f(x_2,y_2^\prime)-f(x_1,y_1)\le C_f\left(\|x_1-x_2\|+\|y_1-y_2^\prime\|\right)\le (\kappa+1) C_f \Vert x_1  -x_2 \Vert,
\end{align*}
This establishes the Lipschitz continuity of $\varphi(x)$.

(f). Without loss of generality, we assume $C_f = 1$,  otherwise we can scale $f(x,y)$ by $C_f$ to prove the result. Because $f(x,y)$ is globally Lipschitz, we can take $K = \BR^{d_x}$ in the proof of \textbf{d}. Then by the same arguments, we can show that $Y^*(x)$ is $C_{\varphi}$-Lipschitz.


\end{proof}

\section{Proof of Theorem \ref{thm:NO}} \label{apx:NO}


Our hard instance is based on the following convex zero-chain. The following function is very similar to the worst-case function given by
\citep{nesterov2018lectures}, Section 2.1.2. The only difference is that the function by \citet{nesterov2018lectures} has an additional term $z_{[q]}^2 / 8$.



\begin{dfn}[Worse-Case Zero-Chain] \label{worse:smooth}
Consider  the family of functions:
\begin{align*}
    h_q(z) = \frac{1}{8} (z_{[1]} -1)^2 +  \frac{1}{8} \sum_{j=1}^{q-1} \left(z_{[j+1]} - z_{[j]} \right)^2. 
\end{align*}
The following properties hold for any $h_q(z)$ with $q \in \mathbb{N}^+$:
\begin{enumerate}[label=\alph*.]
    \item It has a unique minimizer $z^* = \vone $. 
    \item It is convex.
    \item It has $1$-Lipschitz gradients.
    \item It is a first-order zero-chain, \textit{i.e.} for any $z \in \BR^q$, 
    \begin{align*}
        {\rm supp}\{z \} \in \{1,2,\cdots,j \} \Rightarrow {\rm supp}\{\nabla h(z) \} \in \{1,2,\cdots,j+1 \}.
    \end{align*}
\end{enumerate}
\end{dfn}

\begin{proof}
We prove each property one by one.
It can easily be seen that $h_q(z) \ge 0$ for all $z \in \BR^q$ and the equality holds if and only if $z = \vone$.  This proves property \textbf{a}. 
Further, note that $h_q(z)$ is quadratic with Hessian given by
\begin{align*}
    A = \frac{1}{4}
    \begin{bmatrix}
    2 & -1 &  & & \\
    -1 & 2 & -1 & \\
     & -1 & 2 & -1 & \\
     & & \ddots & \ddots & \ddots \\
     & & &  -1 & 1
    \end{bmatrix}.
\end{align*}
As $A$ is diagonally dominant, we  know that $A \succeq O$. This proves property \textbf{b}. 
For any $v \in \BR^q$, 
\begin{align*}
    v^\top A v &= \frac{1}{4} \left[ v_{[1]}^2 + \sum_{j=1}^{q-1} (v_{[j]} - v_{[j+1]})^2 \right]  \\
    &\le \frac{1}{4} \left[ v_{[1]}^2 + \sum_{j=1}^{q-1} (v_{[j]} - v_{[j+1]})^2 + v_{[q]}^2 \right] \\
    &\le \frac{1}{4} \left[ v_{[1]}^2 + \sum_{j=1}^{q-1} 2(v_{[j]}^2 + v_{[j+1]}^2) + v_{[q]}^2 \right] \\
    &\le  \sum_{j=1}^q v_{[j]}^2 = \Vert v \Vert^2.
\end{align*}
This proves property \textbf{c}. Finally property \textbf{d} holds since $A$ is tridiagonal.
\end{proof}

In bilevel problems, it is crucial to find a point $y$ that is close to $Y^*(x)$, instead of just achieving a small optimality gap $g(x,y)- g^*(x)$. However, it is difficult for any first-order algorithms to
``locate'' the minimizers of the function class in Definition \ref{worse:smooth}.
Below, we formalize this observation into a rigorous statement.

\ThmNO*





\begin{proof}
Let $d_x =1$, $d_y = q =2 T K$, $\beta = 1 / \sqrt{q}$ and 
\begin{align*}
    f(x,y) = 2(x+1)^2 r(y), \quad g(y) = \beta^2 h_q(y /\beta ).
\end{align*}
where $h_q(y)$ follows Definition \ref{worse:smooth} and $r(y) = \sum_{j=q/2+1}^{q} \psi(y_{[j]}) $, where $ \psi: \BR \rightarrow \BR $ is 
\begin{align} \label{eq:hermitte-phi}
\begin{split}
    \psi(y) = {\begin{cases}
        \beta^2, & y> 2\beta; \\[1mm]      
        p(y), & \beta < y \le 2\beta; \\[1mm]
        y^2/2, & -\beta \le y \le \beta; \\[1mm]
        p(-y), & -2 \beta \le y < -\beta; \\[1mm]
        \beta^2, & y < -2 \beta; \\[1mm]
    \end{cases}}
\end{split}
\end{align}
where 
\begin{align*}
    p(y) = - \frac{y^5}{2 \beta^3} + \frac{9 y^4}{2 \beta^2} - \frac{31 y^3}{2\beta} + 25 y^2 - 18 \beta y + 5 \beta^2
\end{align*}
is the Hermite interpolating polynomial that satisfies
\begin{align*}
    p(\beta)= \beta^2/2, ~ p'(\beta) = \beta, ~ p''(\beta) = 1 ~~{\rm and} ~~ p(2 \beta) = \beta^2, ~ p'(2 \beta)=0, ~ p''(2\beta) = 0.
\end{align*}
There must exist numerical constants $\gamma_0,\gamma_1,\gamma_2,\gamma_3$ such that
\begin{align*}
    0 < \psi(y) \le \gamma_0 \beta^2, \quad \vert \psi'(y) \vert \le \gamma_1 \beta, \quad \vert \psi''(y) \vert \le \gamma_2,\quad \vert \psi'''(y) \vert \le \gamma_3 / \beta.
\end{align*}
We can verify that $r(y)$ is both bounded and Lipschitz  because
\begin{align*}
   r(y) = \sum_{j=q/2+1}^q \psi(y_{[j]}) \le  \frac{\gamma_0}{2},
\end{align*}
and
\begin{align*}
   \Vert \nabla r(y) \Vert = \sqrt{ \sum_{j=q/2+1}^q \left( \psi'(y_{[j]} \right)^2} \le \frac{\gamma_1}{\sqrt{2}}.
\end{align*}
Furthermore, we can also prove that $r(y)$ has Lipschitz gradients. 
For any $y,y' \in \BR^{d_y}$, 
{\small \begin{align*}
 \Vert \nabla r(y) - \nabla r(y') \Vert = \sqrt{\sum_{j=q/2+1}^q \left( \psi'(y_{[j]}) - \psi'(y'_{[j]}) \right)^2} \le \gamma_2 \sqrt{\sum_{j=q/2+1}^q \left( y_{[j]} - y'_{[j]} \right)^2} \le  \gamma_2 \Vert y - y' \Vert.
\end{align*}}
Note that $ \nabla_x f(x,y) = 4(x+1) r(y)$, $\nabla_y f(x,y) =   2(x+1)^2 \nabla r(y)$.
For any $x \in \fX = [-2,0]$,
\begin{align*}
    &\quad \vert \nabla_x f(x,y) - \nabla_x f(x',y') \vert \\
    &\le  4 \vert x - x' \vert \cdot r(y) + 4 \vert x'  +1 \vert \cdot \vert r(y) - r(y') \vert \\
    &\le 2 \gamma_0 \vert x- x' \vert + 2 \sqrt{2}~ \gamma_1 \Vert y - y' \Vert.
\end{align*}
as well as
\begin{align*}
    &\quad \Vert \nabla_y f(x,y) - \nabla_y f(x',y') \Vert \\
    &\le 2\left( (x+1)^2 - (x'+1)^2 \right) \cdot \Vert \nabla r(y) \Vert + 2 (x'+1)^2 \cdot \Vert \nabla r(y) - \nabla r(y') \Vert \\
    &\le 2 \sqrt{2}~\gamma_1 \vert x -x' \vert + 2 \gamma_2 \Vert y - y' \Vert.
\end{align*}
These two inequalities imply $f(x,y)$ has Lipschitz gradients on $\fX \times \BR^{d_y}$.


Below, prove by induction that $x_t = 0$, and $y_{t,[j]}^k = 0$ for all $j > tK +k$. 

Suppose $x_t = 0$, then $\nabla_y f(x_t,y) =  2 \nabla r(y)$. If we have $y_{t,[j]}^k = 0$ for all $j > tK+k$, then we have $ \nabla r(y)_{[j]} = \psi'(y_{[j]}) = 0$ for all these coordinates $j$. By the property of zero-chain $g(y)$, we also have $ \nabla g(y)_{[j]} =0  $ for all $j > tK +k +1$. This indicates that each inner loop iteration step $k$ increases $y_t^k$ by at most one non-zero coordinate. Since there are at most $q/2$ iterations for $y$ in total, the last $q/2$ coordinates of $y_t^k$ will always remain zero. Since $r(y)$ only depends on the last $q/2$ coordinates, we have $\nabla_x f(x_t,y_t^K) =  4 \sum_{j=q/2}^q \psi(y_{t,[j]}^K) = 0$. Then by the update rule on $x_t$, we have $x_{t+1} = x_t$ remains unchanged. 

The optimal solution of the lower-level function is unique and given by $y^* = \beta \vone$. 
By $r(y^*) = 1/4 $, we know that
$\varphi(x) = (x+1)^2 / 2$. 






\end{proof}

\section{Proof of Theorem \ref{sec:YES}}

First of all, we recall some useful lemmas for PL conditions~\citep{karimi2016linear}.

\begin{lem} \label{lem:PL-EB}
For $\mu$-PL function $h(x): \BR^d \rightarrow \BR$ with Lipschitz gradients, for any $x \in \BR^d$,
\begin{align*}
    \Vert \nabla h(x) \Vert \ge \mu {\rm dist}(x,X^*), ~~{\rm where}~~ X^* = \arg \min_{x \in \BR^d} h(x).
\end{align*}
\end{lem}

\begin{lem} \label{lem:PL-QG}
    For a $\mu$-PL function $h(x): \BR^d \rightarrow \BR$ with Lipschitz gradient, for any $x \in \BR^d$,
    \begin{align*}
        h(x) - \min_{x \in \BR^d} h(x) \ge \frac{\mu}{2} {\rm dist}^2(x,X^*). 
    \end{align*}
\end{lem}

The above two lemmas appear in Theorem 2 \citep{karimi2016linear}. Based on these results, we can prove the following lemma.

\lemsetLip*

\begin{proof}
By Lemma \ref{lem:PL-EB}, for any $y_1 \in Y_{\sigma_1}^*(x_1)$, there exists some $y_2 \in Y_{\sigma_2}^*(x_2)$ such that
     \begin{align*}
        &\quad \mu \Vert y_1 - y_2 \Vert \\
        &\le \Vert \nabla_y h_{\sigma_2}(x_2,y_1) \Vert \\
        &= \Vert \nabla_y h_{\sigma_2}(x_2,y_1) - \nabla_y h_{\sigma_1}(x_1,y_1) \Vert \\
        &\le \Vert \sigma_2 \nabla_y f(x_2,y_1) -   \sigma_1 \nabla_y f(x_1,y_1) \Vert + \Vert \nabla_y g(x_2,y_1) - \nabla_y g(x_1,y_1) \Vert \\
        &\le \Vert \sigma_2 \nabla_y f(x_2,y_1) -   \sigma_1 \nabla_y f(x_2,y_1) \Vert \\
        &\quad + \Vert \sigma_1 \nabla_y f(x_2,y_1) -   \sigma_1 \nabla_y f(x_1,y_1) \Vert +  \Vert \nabla_y g(x_2,y_1) - \nabla_y g(x_1,y_1) \Vert \\
        &\le \vert \sigma_1 - \sigma_2 \vert C_f + (\sigma L_f + L_g) \Vert x_1 - x_2 \Vert.
    \end{align*}
By symmetry, for any $y_2 \in Y_{\sigma_2}^*(x_2)$, there also exists $y_1 \in Y_{\sigma_1}^*(x_1)$ such that that the above inequality holds for $y_2,y_2$. This proves the Pompeiu–Hausdorff continuity.
\end{proof}

\lemDanskin*

\begin{proof}
    It follows the generalized Danskin's theorem proved in \cite{shen2023penalty}. See also Lemma A.2 in \cite{kwon2023penalty}.
\end{proof}

\lemKwon*
\begin{proof}
The proof follows 
Theorem 3.8 \cite{kwon2023penalty}.
The only difference is that Theorem 3.8 \cite{kwon2023penalty} states 
$ \Vert \nabla \varphi_{\sigma}(x) - \varphi(x) \Vert = \fO(\sigma \ell \kappa^5)$. But the additional $\kappa^2$ dependency comes from the perturbation in the multiplier $ {\rm d} \lambda / {\rm d} \sigma \asymp \kappa^2$. Since we only consider the unconstrained case, there is no need to consider the effect of ${\rm d} \lambda / {\rm d} \sigma $, and we can improve the bound to $ \fO(\sigma \ell \kappa^3)$.
\end{proof}
We also recall some technical lemmas from \cite{kwon2023penalty}.

\begin{lem} \label{lem:space}
Suppose $Y_{\sigma}^*(x)$ is Pompeiu–Hausdorff Lipschitz, then for any $y_{\sigma}^*(x) \in Y_{\sigma}^*(x)$,
    \begin{align*}
        {\rm Range}(\nabla_{yx}^2 h_{\sigma}(x,y_{\sigma}^*(x))) &\subseteq {\rm Range}(\nabla_{yy}^2 h_{\sigma}(x,y_{\sigma}^*(x)) ) \\
        \nabla_y f(x,y_{\sigma}^*(x)) &\in {\rm Range}(\nabla_{yy}^2 h_{\sigma}(x,y_{\sigma}^*(x)) ).
    \end{align*}
\end{lem}

\begin{proof}
See Proposition 3.1 \citep{kwon2023penalty}. 
\end{proof}

We remark that Proposition 6 \citep{arbel2022non} also presents a similar argument as the above lemma for Morse-Bott functions.

\begin{lem} \label{lem:nabla-sigma}
Under Assumption \ref{asm:PL},
there exists some $\sigma \in [0,\sigma']$ such that
\begin{align*}
    \nabla \varphi_{\sigma'}(x) = \nabla_x f(x,y_{\sigma}^*(x)) - \nabla_{xy}^2 h_{\sigma}(x,y_{\sigma}^*(x)) \left( \nabla_{yy}^2 h_{\sigma}(x,y_{\sigma}^*(x)) \right)^{\dagger} \nabla_y  f(x,y_{\sigma}^*(x)).
\end{align*}
for any $y_{\sigma}^*(x) \in Y_{\sigma}^*(x)$.
\end{lem}

\begin{proof}
We can express $\varphi_{\sigma'}(x)$ by
\begin{align*}
    \varphi_{\sigma'}(x) = \frac{l(x,\sigma') - l(x,0)}{\sigma'} = \frac{\partial}{\partial \sigma} l(x,\sigma), \quad \exists \sigma \in [0,\sigma'],
\end{align*}
where $l(x;\sigma') = \min_{y} h_{\sigma'}(x,y)$ and we apply the mean-value theorem in the second equality. Taking derivative with respect to $x$ in the above equation yields
\begin{align*}
    \nabla \varphi_{\sigma'}(x) = \dfrac{\partial^2}{\partial x \partial \sigma} l(x,\sigma).
\end{align*}
Finally, we plug in the explicit form of  $\dfrac{\partial^2}{\partial x \partial \sigma} l(x,\sigma)$ by
Theorem 3.2 \citep{kwon2023penalty}.

\end{proof}


The following lemma is a fact from linear algebra.
\begin{lem} \label{lem:UV}
If ${\rm Range}(U^\top) \subseteq {\rm Range}(A^\top)$ and ${\rm Range}(V) \subseteq {\rm Range}(B)$, then
\begin{align*}
    U(A^{\dagger} - B^{\dagger}) V = UA^{\dagger} (B-A) B^{\dagger} V.
\end{align*}
\end{lem}

\begin{proof}
If there exists some matrix $P$ such that $V = BP$, then
\begin{align*}
    V = BP = B B^\dagger B P = BB^\dagger V.
 \end{align*}
 Similarly, if there exists some matrix $Q$ such that $U = QA $, then $ U = U AA^\dagger $.
Combining these two identities completes the proof.
\end{proof}

Under the PL condition, the smallest eigenvalue of Hessian at any minimum is bounded below.
\begin{lem} \label{lem:PL-singular}
For a $\mu$-PL function $h(x): \BR^d \rightarrow \BR$ that is twice differentiable, at any $x^* \in \arg \min_{x \in \BR^d} h(x)$, 
\begin{align*}
    {\lambda}_{\min}^+  \left( \nabla^2 h(x^*) \right) \ge \mu, 
\end{align*}
where ${\lambda}_{\min}^+(\,\cdot\,)$ denotes the smallest non-zero eigenvalue.
\end{lem}

\begin{proof}
 Let $ \lambda_1, \lambda_2, \cdots, \lambda_d$ be the eigenvalues of $\nabla^2 h(x^*)$ in descending order, and $v_1,v_2,\cdots,v_d$ be the corresponding unit eigenvectors which are mutually orthogonal. Let $r$ be the rank of $\nabla^2 h(x^*) $. Then ${\rm Span}(v_1,\cdots,v_r) = {\rm Range}(\nabla^2 h(x^*))$, and  ${\rm Span}(v_{r+1},\cdots,v_d) = {\rm Ker}(\nabla^2 h(x^*))$.

Let $X^* = \arg \min_{x \in \BR^d} h(x)$, $x_t = x^* + t v_r$ and $\hat x_t = \arg \min_{x \in X^*} \Vert x_t - x \Vert$. There exist some coefficients $\alpha_i$ such that $\hat x_t - x^* = \sum_{i=1}^d \alpha_i v_i$. By the Taylor's expansion 
\begin{align*}
    0 &= h(\hat x_t) - h(x^*)  \\
    &= \frac{1}{2}(\hat x_t - x^*)^\top \nabla^2 h(x^*) (\hat x_t - x^*)  + o\left( \Vert \hat x_t-x^* \Vert^2 \right) \\
    &= \frac{1}{2} \sum_{i=1}^r  \lambda_i \alpha_i^2 +  o\left( \Vert \hat x_t-x^* \Vert^2 \right) \\
    &\ge \frac{1}{2} \lambda_r \alpha_r^2  +  o\left( \Vert \hat x_t-x^* \Vert^2 \right).
\end{align*}
By triangle inequality and the definition of $\hat x_t$, we have
\begin{align*}
    \Vert \hat x_t - x^* \Vert \le \Vert x_t - x^* \Vert + \Vert x_t - \hat x_t \Vert \le 2 \Vert x_t - x^* \Vert = 2 t.
\end{align*}
Therefore,
\begin{align*}
     \lambda_r \alpha_r^2 = o\left( \Vert \hat x_t - x^* \Vert^2 \right) =  o\left( t^2 \right).
\end{align*}
On the one hand, Lemma \ref{lem:PL-QG} indicates
\begin{align*}
    &\quad h(x_t) - h(x^*) \\
    &\ge \frac{\mu}{2} \Vert x_t- \hat x_t \Vert^2 \\
    & = \frac{\mu}{2} \Vert x_t- x^* + x^*-  \hat x_t \Vert^2 \\
    &=  \frac{\mu}{2} \left( \left \Vert t v_r - \alpha_r v_r \right \Vert^2 + \left \Vert \sum_{i \ne r} \alpha_i v_i \right \Vert^2   \right) \\
    &\ge \frac{\mu }{2} (t-\alpha_r)^2.
\end{align*}
On the other hand, the Taylor's expansion also indicates
\begin{align*}
    h(x_t) - h(x^*) &= \frac{1}{2}(x_t - x^*)^\top \nabla^2 h(x^*) (x_t - x^*) + o \left( \Vert x_t - x^* \Vert^2 \right) = \frac{\lambda_r}{2} t^2 + o \left( t^2 \right).
\end{align*}
Putting these two hands  together 
\begin{align*}
    \frac{\mu }{2} (t-\alpha_r)^2  \le \frac{\lambda_r}{2} t^2 + o(t^2).
\end{align*}
Using $\alpha_r = o(t)$ and letting $t \rightarrow 0$, we conclude that $\lambda_r \ge \mu$.
\end{proof}

Lemma \ref{lem:Kwon} only claims the existence of $\nabla \varphi(x)$. Below, we give the explicit form of $\nabla \varphi(x)$.

\begin{lem} \label{lem:nabla-phi-form}
Under Assumption \ref{asm:PL},
\begin{align} \label{eq:nabla-phi}
    \nabla \varphi(x) = \nabla_x f(x,y^*(x)) - \nabla_{xy}^2 g(x,y^*(x)) \left( \nabla_{yy}^2 g (x,y^*(x)) \right)^\dagger \nabla_y f(x,y^*(x))
\end{align}
for any $y^*(x) \in Y^*(x)$.
\end{lem}

\begin{proof}
Let $ H(x)$ be the right-hand side of \Eqref{eq:nabla-phi}. Below, we show that $\nabla \varphi(x) = H(x)$.

Recall Lemma \ref{lem:nabla-sigma} that for any $\sigma' \ge 0$ there exists some $\sigma \in [0,\sigma']$
such that
\begin{align*}
    \nabla \varphi_{\sigma'}(x) = \nabla_x f(x,y_{\sigma}^*(x)) - \nabla_{xy}^2 h_{\sigma}(x,y_{\sigma}^*(x)) \left( \nabla_{yy}^2 h_{\sigma}(x,y_{\sigma}^*(x)) \right)^{\dagger} \nabla_y  f(x,y_{\sigma}^*(x)).
\end{align*}
for any $y_{\sigma}^*(x) \in Y_{\sigma}^*(x)$.
Then for any $y^*(x) \in Y^*(x)$, by Lemma \ref{lem:set-Lip}, there exists $y_{\sigma}^*(x) \in Y_{\sigma}^*(x)$ such that $ \Vert y_{\sigma}^*(x) - y^*(x) \Vert \le C_f \sigma/ \mu$. Then by Lemma \ref{lem:space} and Lemma \ref{lem:UV}, we have
\begin{align*}
    &\quad \Vert \nabla \varphi_{\sigma'}(x) - H(x) \Vert \\
    &\le \Vert \nabla_x f(x,y_{\sigma}^*(x)) - \nabla_x f(x,y^*(x))) \Vert \\
    &\quad +   \left \Vert \left( \nabla_{xy}^2 h_{\sigma} (x,y_{\sigma}^*(x)) - \nabla_{xy}^2 g (x,y^*(x))\right) \left( \nabla_{yy}^2 h_{\sigma}(x,y_{\sigma}^*(x)) \right)^{\dagger} \nabla_y f(x,y^*(x))) \right \Vert \\
    &\quad +  \Big \Vert \nabla_{xy}^2 g (x,y^*(x)) \left( \nabla_{yy}^2 g (x,y^*(x)) \right)^{\dagger} \\
    &~~~~~~~~~~~~~~~~~~~~~~~~~
    \left( \nabla_{yy}^2 g (x,y^*(x)) - \nabla_{yy}^2 h_{\sigma} (x,y_{\sigma}^*(x)) \right) \left( \nabla_{yy}^2h_{\sigma} (x,y_{\sigma}^*(x)) \right)^{\dagger}
    \nabla_y f(x,y_{\sigma}^*(x)) 
    \Big \Vert \\
    &\quad + \left \Vert \nabla_{xy}^2 g(x,y^*(x))  \left(  \nabla_{yy}^2 g (x,y^*(x)) \right)^{\dagger} 
    \left( \nabla_y f(x,y^*(x)) - \nabla_y f(x,y_{\sigma}^*(x))  \right)
    \right \Vert
\end{align*}
We then use Lemma \ref{lem:PL-singular} to have a further upper bound as
\begin{align*}
     \Vert \nabla \varphi_{\sigma'}(x) - H(x) \Vert \le \frac{\sigma C_f}{\mu} \left( 1+ \frac{L_g}{\mu} \right)
     \left(2 L_f + \frac{\rho_g C_f}{\mu} \right).
\end{align*}
Taking $\sigma' \rightarrow 0^+ $, we conclude
\begin{align*}
    \nabla \varphi(x) = \lim_{\sigma' \rightarrow 0^+ } \nabla \varphi_{\sigma'}(x) = H(x).
\end{align*}

\end{proof}
\lemgradLip*
\begin{proof}
Invoking Lemma \ref{lem:set-Lip}, there exists $y_1 \in Y^*(x_1)$ and $y_2 \in Y^*(x_2)$ such that $\Vert y_1 - y_2 \Vert \le L_g / \mu$. Then by Lemma \ref{lem:space} and Lemma \ref{lem:UV}, we have
\begin{align*}
    &\quad \Vert \nabla \varphi(x_1) - \nabla \varphi(x_2) \Vert \\
    &\le \Vert \nabla_x f(x_1,y_1) - \nabla_x f(x_2,y_2)) \Vert \\
    &\quad +   \left \Vert \left( \nabla_{xy}^2 g (x_1,y_1) - \nabla_{xy}^2 g (x_2,y_2)\right) \left( \nabla_{yy}^2 g(x_1,y_1) \right)^{\dagger} \nabla_y f(x_1,y_1)) \right \Vert \\
    &\quad +  \Big \Vert \nabla_{xy}^2 g (x_2,y_2) \left( \nabla_{yy}^2 g (x_2,y_2) \right)^{\dagger} \\
    &~~~~~~~~~~~~~~~~~~~~~~~~~
    \left( \nabla_{yy}^2 g (x_2,y_2) - \nabla_{yy}^2 g (x_1,y_1) \right) \left( \nabla_{yy}^2 g (x_1,y_1) \right)^{\dagger}
    \nabla_y f(x_1,y_1) 
    \Big \Vert \\
    &\quad + \left \Vert \nabla_{xy}^2 g(x_2,y_2)  \left(  \nabla_{yy}^2 g (x_2,y_2) \right)^{\dagger} 
    \left( \nabla_y f(x_2,y_2) - \nabla_y f(x_1,y_1)  \right)
    \right \Vert.
\end{align*}
Further invoking Lemma \ref{lem:PL-singular},
\begin{align*}
    \Vert \nabla \varphi(x_1) - \nabla \varphi(x_2) \Vert \le \left( L_f + \frac{C_f \rho_g}{\mu} \right) \left( 1 + \frac{L_g}{\mu} \right) \left(1 + \frac{L_g}{\mu} \right)  \Vert x_1 - x_2 \Vert.
\end{align*}
\end{proof}

The following lemma shows linear-convergence of gradient descent on PL functions.
\begin{lem} \label{lem:PL-GD}
Suppose $h(x): \BR^d \rightarrow \BR$ is $\alpha$-PL and has $\beta$-Lipschitz gradients. Consider the following update of gradient descent:
\begin{align*}
    x_{t+1} = x_t - \frac{1}{\beta} \nabla h(x_t).
\end{align*}
Let $X^* = \arg \min_{x \in \BR^d} h(x)$ and $h^* = \min_{x \in \BR^d} h(x)$. Then it holds that
\begin{align*}
    {\rm dist}^2 (x_T,X^*) \le \left(1 - \frac{\alpha}{\beta} \right)^T \frac{\beta}{\alpha} {\rm dist}^2 (x_0,X^*).
\end{align*}
\end{lem}
\begin{proof} We first prove the linear convergence on the sub-optimality gap.
\begin{align*}
    h(x_{t+1}) - h^* &\le h(x_t) - h^* + \nabla h(x_t)^\top (x_{t+1} - x_t) + \frac{\beta}{2} \Vert x_{t+1} - x_t \Vert^2 \\
    &= h(x_t) - h^* - \frac{1}{2 \beta} \Vert \nabla h(x_t) \Vert^2 \\
    &\le \left( 1 - \frac{\alpha}{\beta} \right) (h(x_t) - h^*).
\end{align*}
    Telescope over $t = 0,\cdots, T-1$
    \begin{align*}
        h(x_{T}) - h^* \le \left( 1 - \frac{\alpha}{\beta}\right)^T (h(x_0) - h^*).
    \end{align*}
We complete the proof by noting that
\begin{align} \label{eq:note-QG}
    h(x) - h^* \le \frac{\beta}{2} {\rm dist}^2(x,X^*) \quad {\rm and} \quad h(x)  - h^* \ge \frac{\alpha}{2} {\rm dist}^2(x,X^*).
\end{align}
\end{proof}

Then we can easily show that $\nabla \varphi_{\sigma}(x)$ can be efficiently approximated in logarithmic time. Combining both the outer and inner iterations yields the following result.

\ThmYES* 

\begin{proof}
Let $L$ be the gradient Lipschitz constant of $\varphi(x)$. Let $\eta \le 1/(2L)$, then
\begin{align} \label{eq:des}
\begin{split}
    \varphi(x_{t+1}) &\le \varphi(x_t) + \langle \nabla \varphi(x_t) , x_{t+1} - x_t
    \rangle + \frac{L}{2} \Vert x_{t+1 } - x_t \Vert^2 \\
    &= \varphi(x_t) -
    \frac{\eta}{2} \Vert \nabla \varphi(x_t) \Vert^2 -
    \left( \frac{\eta}{2} - \frac{\eta^2 L}{2} \right) \Vert \hat \nabla \varphi(x_t) \Vert^2 + \frac{\eta}{2} \Vert \hat  \nabla \varphi(x_t) - \nabla \varphi(x_t) \Vert^2 \\
    &\le \varphi(x_t)
    - \frac{\eta}{2} \Vert \nabla \varphi(x_t) \Vert^2 - 
    \frac{1}{4\eta} \Vert x_{t+1} - x_{t} \Vert^2
    + \frac{\eta}{2} \Vert  \hat \nabla \varphi(x_t) - \nabla \varphi_{\sigma}(x_t) \Vert^2 +\fO(\eta \epsilon^2).
\end{split}
\end{align}
Note that
\begin{align}  \label{eq:des-dist}
    \Vert \hat \nabla \varphi(x_t) - \nabla \varphi_{\sigma}(x_t) \Vert \le \frac{2  L_g}{\sigma} {\rm dist}( y_{t}^K, Y_{\sigma}^*(x_t))  + \frac{L_g}{\sigma} {\rm dist}(z_{t}^K, Y^*(x_t)).
\end{align}
Then by Lemma \ref{lem:PL-GD}, we have
\begin{align} \label{eq:err}
    \Vert \hat \nabla \varphi (x_t)- \nabla \varphi_{\sigma}(x_t) \Vert^2 \le \frac{8 L_g^3}{\mu \sigma^2 } \exp\left( - \frac{\mu K}{2 L_g} \right) \left( {\rm dist}^2 (y_t^K, Y_{\sigma}^*(x_t)) 
    + {\rm dist}^2 (z_t^K, Y^*(x_t))
    \right).
\end{align}
By Young's inequality and Lemma \ref{lem:set-Lip},  
\begin{align*}
    {\rm dist}^2\left(y_{t+1}^0, Y_{\sigma}^*(x_{t+1})\right) & \le 2 {\rm dist}^2\left( y_t^K, Y_{\sigma}^*(x_{t})\right)  + 2 {\rm dist}^2 \left(Y_{\sigma}^*(x_{t+1}) , Y_{\sigma}^*(x_t)\right)  \\
    &\le \frac{4 L_g}{\mu} \exp \left( - \frac{\mu K}{2 L_g} \right) {\rm dist}^2 \left( y_t^0, Y_{\sigma}^*(x_t) \right)  + \frac{8 L_g^2}{\mu^2} \Vert x_{t+1} - x_t \Vert^2,
\end{align*}
Similarly, we can derive the recursion about ${\rm dist}^2 \left( z_t^0, Y^*(x_t) \right)$. 

Put them together and let
\begin{align*}
    K \ge  \frac{2 L_g}{\mu} \log \left( \frac{8 L_g}{\mu} \right),
\end{align*}
we have
\begin{align*}
    \delta_{t+1} 
    &\le \frac{1}{2} \delta_t +\frac{16 L_g^2}{\mu^2} \Vert x_{t+1} - x_t \Vert^2,
\end{align*}
where we define $\delta_t := {\rm dist}^2 \left(y_{t}^0,  Y_{\sigma}^*(x_{t}) \right)  + {\rm dist}^2 \left(z_{t}^0, Y^*(x_{t}) \right)$. Telescoping over $t$ yields
\begin{align*}
    \delta_t
    &\le \underbrace{\left( \frac{1}{2} \right)^t \delta_0 + \frac{16 L_g^2}{\mu^2} \sum_{j=0}^{t-1}  \left( \frac{1}{2}\right)^{t-1-j} \Vert x_{j+1} - x_j \Vert^2}_{:=(*)}.
\end{align*}
Plug into \Eqref{eq:err}, which, in conjunction with \Eqref{eq:des}, yields that
\begin{align*}
     \varphi(x_{t+1}) &\le \varphi(x_t) 
    - \frac{\eta}{2} \Vert \nabla \varphi(x_t) \Vert^2 - 
    \frac{1}{4\eta} \Vert x_{t+1} - x_{t} \Vert^2 + 4 \eta \times \underbrace{ \frac{L_g^3}{\mu \sigma^2} \exp \left( - \frac{\mu K}{2L_g}\right)}_{:=\gamma} \times (*) + \fO(\eta \epsilon^2).
\end{align*}
Telescoping over $t$ further yields
\begin{align} \label{eq:final}
\begin{split}
    \frac{\eta}{2}\sum_{t=0}^{T-1} \Vert  \nabla \varphi(x_t) \Vert^2 &\le \varphi(x_0) - \inf_{x \in \BR^{d_x}} \varphi(x) + 8 \eta \gamma \delta_0 \\
    &\quad - \left( \frac{1}{4 \eta} - \frac{148 \eta \gamma L_g^2}{\mu^2} \right) \sum_{t=0}^{T-1} \Vert x_{t+1} - x_t \Vert^2 +\fO(\eta \epsilon^2).
\end{split}
\end{align}

Let $K = \fO(\kappa \log( \nicefrac{\kappa}{\sigma})) = \fO(\kappa \log(\nicefrac{\ell \kappa}{\epsilon}))$ such that $\gamma$ is sufficiently small with
\begin{align*}
    \gamma \le \min \left\{ \frac{\mu^2}{1184 \eta^2 L_g^2}, \frac{1}{8 \eta}\right\}.
\end{align*}
Then we have,
\begin{align*}
    \frac{1}{T} \sum_{t=0}^{T-1} \Vert \nabla \varphi(x_t) \Vert^2 \le  \frac{2}{\eta T} \left(\Delta + \delta_0 \right) + \fO(\epsilon^2).
\end{align*}
By Lemma \ref{lem:set-Lip} we know that for any 
\begin{align*}
    {\rm dist}^2 (y_0, Y_{\sigma}^*(x)) \le 2  {\rm dist}^2 (y_0, Y^*(x)) + 2  {\rm dist}^2 \left(Y_{\sigma}^*(x), Y^*(x)\right)= \fO(R). 
\end{align*}
Hence $\delta_0=R$ is also bounded.
This concludes the proof.
    
\end{proof}

\section{Proof of Theorem \ref{thm:stoc-F2BA}}

\begin{lem} \label{lem:PL-SGD}
Suppose $h(x): \BR^d \rightarrow \BR$ is  $\alpha$-PL and has $\beta$-Lipschitz gradients. Consider the following update of stochastic gradient descent:
\begin{align*}
    x_{t+1} = x_t - \frac{1}{\beta} \nabla h(x_t;\fB_t),
\end{align*}
where the mini-batch gradient satisfies 
\begin{align*}
    \BE_{\fB_t} \left[ \nabla h(x_t;\fB_t) \right] = \nabla h(x_t) , \quad \BE_{\fB_t} \Vert \nabla h(x_t;\fB_t) - \nabla h(x_t) \Vert^2 \le \frac{M^2}{B}.
\end{align*}
Let $X^* = \arg \min_{x \in \BR^d} h(x)$ and $h^* = \min_{x \in \BR^d} h(x)$. Then it holds that
\begin{align*}
    \BE \left[{\rm dist}^2 (x_T,X^*)\right] \le \left(1 - \frac{\alpha}{\beta} \right)^T \frac{\beta}{\alpha} {\rm dist}^2 (x_0,X^*) + \frac{M^2}{\alpha^2 B} . 
\end{align*}
\end{lem}

\begin{proof}
    The proof is similar to the deterministic case. Conditional on $x_t$, we have
    \begin{align*}
        \BE \left[ h(x_{t+1}) - h^*\right] &\le \BE \left[ h(x_t) - h^* + \nabla h(x_t)^\top (x_{t+1} - x_t) + \frac{\beta}{2} \Vert x_{t+1} - x_t \Vert^2 \right]\\
        &= \BE \left[ h(x_t) - h^* - \frac{1}{2 \beta} \Vert \nabla h(x_t) \Vert^2 + \frac{1}{2 \beta} \Vert \nabla h(x_t; \fB_t) - \nabla h(x_t) \Vert^2 \right] \\
        &\le \left( 1 - \frac{\alpha}{\beta}\right) \left( h(x_t) - h^* \right) + \frac{M^2}{2 \beta B}.
    \end{align*}
Telescope,
\begin{align*}
    \BE \left[ h(x_{t+1}) - h^*\right] \le \left(1 - \frac{\alpha}{\beta} \right)^T  \left( h(x_0) - h^* \right)  + \frac{M^2}{2 \alpha B}.
\end{align*}
We conclude the proof by using \Eqref{eq:note-QG}.

\end{proof}

\ThmStoc*

\begin{proof}
Define $\delta_t := \BE \left[ {\rm dist}^2 \left(y_{t}^0,  Y_{\sigma}^*(x_{t}) \right)  + {\rm dist}^2 \left(z_{t}^0, Y^*(x_{t}) \right) \right]$.
 By  \Eqref{eq:des-dist}, letting
 \begin{align*}
 {\rm dist}^2\left( y_t^K, Y_{\sigma}^*(x_{t})\right) +  {\rm dist}^2\left( z_t^K, Y^*(x_{t})\right)
  \le  \fO\left(\frac{\sigma^2 \epsilon^2}{ L_g^2} \right)      
 \end{align*}
ensures $ \BE \Vert \hat \nabla \varphi(x_t) - \nabla \varphi_{\sigma}(x_t) \Vert^2 \le \fO(\epsilon)$.
Then telescoping over \Eqref{eq:des}  shows that one  can find an $\epsilon$-stationary point of $\varphi(x)$ with $T = \fO(\epsilon^{-2})$ outer-loop iterations.
By Lemma \ref{lem:PL-SGD}, it suffices to set the parameters in the inner loop as \Eqref{eq:para-stoc-F2BA}.
But $K_t$ requires the knowledge of $\delta_t$.
Next, we bound $\delta_t$ via a recursion which allows us to get rid of the prior knowledge of $\delta_t$.
By Lemma \ref{lem:PL-SGD} and Lemma \ref{lem:set-Lip},  
\begin{align*}
 &\quad {\rm dist}^2\left(y_{t+1}^0, Y_{\sigma}^*(x_{t+1})\right) \\
    & \le 2 {\rm dist}^2\left( y_t^K, Y_{\sigma}^*(x_{t})\right)  + 2 {\rm dist}^2 \left(Y_{\sigma}^*(x_{t+1}) , Y_{\sigma}^*(x_t)\right)  \\
    &\le \frac{4 L_g}{\mu} \exp \left( - \frac{\mu K_t}{2 L_g} \right) {\rm dist}^2 \left( y_t^0, Y_{\sigma}^*(x_t) \right) + \frac{8 L_g^2}{\mu^2} \Vert x_{t+1} - x_t \Vert^2 + \fO \left( \frac{\sigma^2 \epsilon^2}{L_g^2} \right).
\end{align*}
Similarly, we can derive the recursion about ${\rm dist}^2 \left( z_t^0, Y^*(x_t) \right)$. 
Put them together and let
\begin{align*}
    K_t \ge  \frac{2 L_g}{\mu} \log \left( \frac{8 L_g}{\mu} \right),
\end{align*}
we can get \Eqref{eq:delta_t}.
Telescoping over $t$,
\begin{align*}
    \sum_{t=0}^{T-1} \delta_t \le 2 \delta_0 + \frac{32 L_g^2}{\mu^2} \sum_{t=0}^{T-1} \Vert x_{t+1} - x_t \Vert^2 + \fO \left( \frac{\sigma^2 \epsilon^2}{L_g^2} \right),
\end{align*}
where on the right-hand side we know $ \sum_{t=0}^{T-1} \Vert x_{t+1} - x_t \Vert$ must also be bounded by  \Eqref{eq:des}.

With this recursion, the total iterations can be bounded above by
\begin{align*} 
\sum_{t=0}^{T-1} K_t \le \sum_{t=0}^{T-1} \frac{2 L_g}{\mu} \log \left( \frac{32 L_g^3 \delta_t}{\mu \sigma^2 \epsilon^2} \right)   \le \frac{2 L_g T}{\mu} \log \left( \frac{ 32 L_g^3\sum_{t=0}^{T-1}\delta_t}{\mu \sigma^2 \epsilon^2 T} \right).
\end{align*}
And the total number of stochastic oracle calls is $B$ times the above bound for iterations.
\end{proof}

\end{document}